\documentclass[a4paper,11pt]{amsart}
\usepackage[T1]{fontenc}
\usepackage{amssymb}
\usepackage{tabularx}
\usepackage[ansinew]{inputenc}
\usepackage[english,
francais]{babel}
\usepackage{color}
\usepackage{xcolor}
\usepackage{graphicx}
 
\usepackage[active]{srcltx} 
\usepackage{amssymb,amsthm,amsfonts,amssymb,euscript,mathrsfs
,
}
\usepackage{amsxtra}
\usepackage{amscd}
\usepackage{float}

\ifx \undefined \cftil \def \cftil#1{\~#1}\fi

\renewcommand{\and}{et}



 \DeclareMathAlphabet{\got}{U}{euf}{m}{n}     
\DeclareMathAlphabet{\mat}{U}{msb}{m}{n}     
\DeclareMathAlphabet{\mathbold}{OML}{cmm}{bx}{it} 

\newtheorem{theo}{Th\'eor\`eme}[section]

\newtheorem{lem}{Lemme}[section]

\newtheorem{defis}{D\'efinitions}[section]
\newtheorem{cor}{Corollaire}[section]
\newtheorem{rem}{Remarque}[section]

\newtheorem{Ex}{Exemple}[section]

\newcommand{\cur}{\mathscr}

\makeatletter
\providecommand*{\diff}%
{\@ifnextchar^{\DIfF}{\DIfF^{}}}
\def\DIfF^#1{%
\mathop{\mathrm{\mathstrut d}}%
\nolimits^{#1}\gobblespace}
\def\gobblespace{%
\futurelet\diffarg\opspace}
\def\opspace{%
\let\DiffSpace\!%
\ifx\diffarg(%
\let\DiffSpace\relax
\else
\ifx\diffarg[%
\let\DiffSpace\relax
\else
\ifx\diffarg\{%
\let\DiffSpace\relax
\fi\fi\fi\DiffSpace}

\providecommand*{\eu}%
{\ensuremath{\mathrm{e}}}

\providecommand*{\iu}%
{\ensuremath{\mathrm{i}}}

\DeclareMathOperator{\Ad }{Ad }

\begin{document}
\title{indice des sous-alg\`ebres biparaboliques d'une alg\`ebre de Lie simple classique}

\date{\today}
\selectlanguage{french}
\author{Meher Bouhani}
\address{ 
Universit\'{e} Tunis El-Manar\\
Facult\'{e} des Sciences de Tunis\\
D\'{e}partement de Math\'{e}matiques\\
Campus Universitaire\\
2092 El-Manar
 \\    Tunis, Tunisie}
\email{Meher.Bouhani@math.univ-poitiers.fr}

\maketitle 

\selectlanguage{english} 
\begin{abstract}
We generalize the results in \cite{MB1} giving a reduction algorithm allowing to compute the index of seaweed subalgebras of classical simple Lie algebras. We thus are able to obtain the index of some interesting families of seaweed subalgebras and to give new examples of large classes of Frobenius Lie algebras among them.
\end{abstract}

\selectlanguage{french} 
\begin{abstract}
On g\'en\'eralise les r\'esultats de \cite{MB1} en donnant un algorithme de r\'eduction pour le calcul de l'indice des sous-alg\`ebres biparaboliques d'une alg\`ebre de Lie simple classique qui permet d'obtenir l'indice de certaines classes int\'eressantes de ces sous-alg\`ebres et d'en d\'eduire en particulier celles parmi elles qui sont des sous-alg\`ebres de Frobenius.  
\end{abstract}

\section{Introduction}
Soit $\mathfrak{g}$ l'alg\`ebre de Lie d'un groupe de Lie alg\'ebrique complexe $\mathbold{G}$ et $\mathfrak{g}^{\ast}$ son dual. Pour $f\in \mathfrak{g}^{*}$, on note  
$\mathfrak{g}_{f}$ le stabilisateur de $f$ pour l'action coadjointe. On appelle indice de $\mathfrak{g}$ et on note $\chi[\mathfrak{g}]$ la dimension minimale de $\mathfrak{g}_{f}$ lorsque $f$ parcourt $\mathfrak{g}^{\ast}$. Si $\chi[\mathfrak{g}]=0$, $\mathfrak{g}$ est dite une alg\`ebre de Frobenius.\\

Dans toute la suite, les groupes et les alg\`ebres de Lie consid\'er\'es sont alg\'ebriques d\'efinis sur le corps des complexes. Pour toute paire $(r,s)$ d'entiers naturels, on note $r[s]$ le reste de la division euclidienne de $r$ par $s$ et $r \wedge s$ le plus grand commun diviseur de $r$ et $s$. Pour $\underline{a}=(a_{1},\ldots,a_{k})\in\mathbb{N}^{k}$, on pose $|\underline{a}|:=a_{1}+\cdots+a_{k}$. \\

Soient $n\in\mathbb{N}^{\times},\;\underline{a}=(a_{1},\ldots,a_{k})$ et
  $\underline{b}=(b_{1},\ldots,b_{t})$ deux compositions v\'erifiant $|\underline{a}|\leq n$ et $|\underline{b}|\leq n$. On associe \`a la paire $(\underline{a},\underline{b})$ une unique sous-alg\`ebre biparabolique de $\mathfrak{sp}(2n)$ (resp. $\mathfrak{so}(2n+1)$, $\mathfrak{so}(2n)$) que l'on note $\mathfrak{q}^{C}_{n}(\underline{a}\mid\underline{b})$ (resp. $\mathfrak{q}^{B}_{n}(\underline{a}\mid\underline{b})$, $\mathfrak{q}^{D}_{n}(\underline{a}\mid\underline{b})$), et toutes les sous-alg\`ebres biparaboliques de $\mathfrak{sp}(2n)$ (resp. $\mathfrak{so}(2n+1)$, $\mathfrak{so}(2n)$) sont ainsi obtenues \`a conjugaison pr\`es par le groupe adjoint connexe de $\mathfrak{sp}(2n)$ (resp. $\mathfrak{so}(2n+1)$, $\mathfrak{so}(2n)$). Cependant, dans le cas de type $D$, comme expliqu\'e dans la section 4, on peut supposer que, si $|\underline{a}|=n$ (resp.  $|\underline{b}|=n$), alors $a_{k}>1$ (resp. $b_{t}>1$). Si $|\underline{a}|=|\underline{b}|= n$, on associe \`a la paire $(\underline{a},\underline{b})$ une unique sous-alg\`ebre biparabolique de $\mathfrak{gl}(n)$ (\`a conjugaison pr\`es par le groupe adjoint connexe de $\mathfrak{gl}(n)$) que l'on note $\mathfrak{q}^{A}(\underline{a}\mid\underline{b})$ (voir \cite{MB}, \cite{meander.C} et \cite{meander.D}).\\
 Pour $\underline{a}=(a_{1},\ldots,a_{k})\in\mathbb{N}^{k}$ et
  $\underline{b}=(b_{1},\ldots,b_{t})\in\mathbb{N}^{t}$ tels que $|\underline{a}|\leq n$ et $|\underline{b}|\leq n$, on d\'esigne par $\tilde{\underline{a}}$ (resp. $\tilde{\underline{b}}$) la suite obtenue de $\underline{a}$ (resp. $\underline{b}$) en supprimant les termes nuls et on convient que $\mathfrak{q}^{A}(\underline{a}\mid\underline{b})=\mathfrak{q}^{A}(\tilde{\underline{a}}\mid\tilde{\underline{b}})$ si $|\underline{a}|=|\underline{b}|=n$  et $\mathfrak{q}^{I}_{n}(\underline{a}\mid\underline{b})=\mathfrak{q}^{I}_{n}(\tilde{\underline{a}}\mid\tilde{\underline{b}})$, $I=B,\;C\;ou\;D$.\\
  
 Pour $n >1$, soit $\Xi_{n}$ l'ensemble des paires de compositions $(\underline{a}=(a_{1},\ldots,a_{k}),\underline{b}=(b_{1},\ldots,b_{t}))$ qui v\'erifient : $|\underline{a}|=n$, $|\underline{b}|=n-1$ et $a_{k}>1$ ou $|\underline{b}|=n$, $|\underline{a}|=n-1$ et $b_{t}>1$. \\

Dans \cite{D.K}, Dergachev et Kirillov associent \`a chaque sous-alg\`ebre biparabolique $\mathfrak{q}^{A}(\underline{a}\mid\underline{b})$ de $\mathfrak{gl}(n)$ un graphe appel\'e m\'eandre de $\mathfrak{q}^{A}(\underline{a}\mid\underline{b})$ et not\'e $\Gamma^{A}(\underline{a}\mid\underline{b})$. Il est construit de la mani\`ere suivante : on place $n$ points cons\'ecutifs, appel\'es sommets de $\Gamma^{A}(\underline{a}\mid\underline{b})$ et num\'erot\'es de $1$ \`a $n$, sur une droite horizontale, puis on relie au dessous (resp. au dessus) de cette droite par un arc toute paire de sommets distincts de la forme $(a_{1}+\cdots+a_{i-1}+j_{i},a_{1}+\cdots+a_{i}-j_{i}+1),\;1\leq j_{i}\leq a_{i},\;1\leq i\leq k$ (resp. $(b_{1}+\cdots+b_{i-1}+j_{i},b_{1}+\cdots+b_{i}-j_{i}+1),\;1\leq j_{i}\leq b_{i},\;1\leq i\leq t$). Au moyen des composantes connexes de ce graphe, qui sont des cycles et des segments, les auteurs d\'ecrivent  l'indice de $\mathfrak{q}^{A}(\underline{a}\mid\underline{b})$ (voir th\'eor\`eme \ref{thm1}). Ce r\'esultat a \'et\'e g\'en\'eralis\'e au cas des sous-alg\`ebres biparaboliques de $\mathfrak{sp}(2n)$ de deux mani\`eres ind\'ependantes dans \cite{C.meanders} et \cite{meander.C}, et au cas des sous-alg\`ebres biparaboliques de $\mathfrak{so}(n)$ dans \cite{meander.C} et \cite{meander.D}. Dans \cite{meander.C} et \cite{meander.D}, les auteurs  associent \`a la sous-alg\`ebre biparabolique $\mathfrak{q}^{I}_{n}(\underline{a}\mid\underline{b})$ un m\'eandre not\'e $\Gamma^{I}_{n}(\underline{a}\mid\underline{b})$,  o\`u $\;I=B,\;C\;ou\;D$. Lorsque $I=B,\;C\;ou\;I=D$ et $(\underline{a}\mid\underline{b})\notin\Xi_{n}$, le m\'eandre $\Gamma^{I}_{n}(\underline{a}\mid\underline{b})$ v\'erifie $\Gamma^{I}_{n}(\underline{a}\mid\underline{b})=\Gamma^{A}(a_{1},\ldots,a_{k},2(n-|\underline{a}|),a_{k},\ldots,a_{1}\mid b_{1},\ldots,b_{t},2(n-|\underline{b}|),b_{t},\ldots,b_{1} )$. Lorsque $I=D$ et $(\underline{a}\mid\underline{b})\in\Xi_{n}$, le m\'eandre $\Gamma^{D}_{n}(\underline{a}\mid\underline{b})$ poss\`ede deux arcs qui se croisent (voir section 4).\\

Dans \cite{C.meanders}, les auteurs donnent une formule de l'indice de la sous-alg\`ebre biparabolique $\mathfrak{q}_{n}^{C}(a,b\mid c)$ lorsque $|a+b-c|=1$ ou $2$ leur permettant de d\'eterminer la famille des sous--alg\`ebres biparaboliques de Frobenius qui sont de la forme $\mathfrak{q}_{n}^{C}(a,b\mid c),\;(a,b,c)\in(\mathbb{N}^{\times})^{3}$. Dans ce travail, nous donnons une formule de l'indice des sous-alg\`ebres biparaboliques de $\mathfrak{sp}(2n)$ (resp. $\mathfrak{so}(2n+1)$, $\mathfrak{so}(2n))$ qui sont de la forme $\mathfrak{q}_{n}^{C}(a,b\mid c)$ (resp. $\mathfrak{q}_{n}^{B}(a,b\mid c)$, $\mathfrak{q}_{n}^{D}(a,b\mid c)$), pour tout $(a,b,c)\in(\mathbb{N}^{\times})^{3}$ (resp. $(a,b,c)\in(\mathbb{N}^{\times})^{3}$, $(a,b,c)\in(\mathbb{N}^{\times})^{3}$ et $b>1$ si $a+b=n$) (voir th\'eor\`emes \ref{I3} et \ref{I'3}). Plus pr\'ecis\'ement, nous montrons le th\'eor\`eme suivant : 

\begin{theo} Soient $a,b,c,n\in\mathbb{N}^{\times}$ tels que $s:=\max(a+b,c)\leq n$. On pose $p=(a+b)\wedge(b+c)$ et $r=|a+b-c|$. Alors

 \begin{itemize}
\item[1)] \begin{itemize}
\item[i)]Si $p>r$, on a $\chi(\mathfrak{q}^{B}_{n}(a,b\mid c))=\chi[\mathfrak{q}^{C}_{n}(a,b\mid c)]=p-r+[\frac{r}{2}]+n-s$
\item[ii)] Si $p\leq r$, on a \\
$\chi[\mathfrak{q}^{B}_{n}(a,b\mid c)]=\chi[\mathfrak{q}^{C}_{n}(a,b\mid c)]=\begin{cases}[\frac{r}{2}]+n-s\;\text{si $p$ et $r$ sont de m\^eme pari\'et\'e}\\
[\frac{r}{2}]-1+n-s\;\text{sinon}\\
\end{cases}$\\
\end{itemize}
\item[2)] Soit $\Gamma^{D}_{n}(a,b\mid c)$ le m\'eandre de $\mathfrak{q}^{D}_{n}(a,b\mid c)$, Alors
\begin{itemize}
\item[i)] Si $((a,b),c)\notin\Xi_{n}$, on a $\chi[\mathfrak{q}^{D}_{n}(a,b\mid c)]=\chi[\mathfrak{q}^{C}_{n}(a,b\mid c)]+\epsilon$, o\`u $\epsilon$ est donn\'e par:$\\$
$ \epsilon=\begin{cases} 0\;\;\;\;si\;r\;\text{est un entier pair} \\
1\;\;\;\;si\;r\;\text{est un entier impair},\;s=n\;\text{et de plus l'arc de }\Gamma^{D}_{n}(a,b\mid c)\;\text{joignant les }\\
 \;\;\;\;\;\;sommets \;n\;et\;n+1\;est\; \text{un arc d'un segment de}\;\Gamma^{D}_{n}(a,b\mid c)\\
-1\;\text{dans les cas restants}\\
\end{cases}$\\
\item[ii)] Si $((a,b),c)\in\Xi_{n}$, on a 
$\chi[\mathfrak{q}_{n}^{D}(a,b\mid c)]=|(a\wedge n)-2|$
\end{itemize}
\end{itemize}
\end{theo}

En particulier, nous caract\'erisons les alg\`ebres de Frobenius de cette famille (voir corollaires \ref{J3} et \ref{J'3}).\\

Pour une sous-alg\`ebre biparabolique $\mathfrak{q}^{A}(\underline{a}\mid\underline{b})$ de $\mathfrak{gl}(n)$, on pose \\
$\Psi[\mathfrak{q}^{A}(\underline{a}\mid\underline{b})]=\begin{cases}\chi[\mathfrak{q}^{A}(\underline{a}\mid\underline{b})]\;\text{si le sommet $n$ appartient \`a un segment de }\Gamma^{A}(\underline{a}\mid\underline{b})\\
\chi[\mathfrak{q}^{A}(\underline{a}\mid\underline{b})]-2\;\text{sinon}\\
\end{cases}$\\

Dans \cite{MB1}, nous avons donn\'e un algorithme de r\'eduction permettant le calcul de l'indice des sous-alg\`ebres biparaboliques $\mathfrak{q}^{A}(\underline{a},\underline{b})$ de $\mathfrak{gl}(n)$. Dans cet article, nous g\'en\'eralisons ce r\'esultat au cas des sous-alg\`ebres biparaboliques $\mathfrak{q}^{C}_{n}(\underline{a},\underline{b})$ (resp. $\mathfrak{q}^{B}_{n}(\underline{a},\underline{b})$, $\mathfrak{q}^{D}_{n}(\underline{a},\underline{b})$) de $\mathfrak{sp}(2n)$ (resp. $\mathfrak{so}(2n+1)$, $\mathfrak{so}(2n)$) (voir th\'eor\`emes \ref{thm0'}, \ref{thm r1} et \ref{thm r2}). Nous montrons d'abord qu'on peut se ramener au cas $\underline{a}=(t)$ et $|\underline{b}|\leq t\leq n$, $t\in\mathbb{N}^{\times}$. Ensuite, nous donnons les deux th\'eor\`emes suivants  

 \begin{theo} Soient $t\in\mathbb{N}^{\times}$ et $\underline{a}=(a_{1},\ldots,a_{k})$ une composition v\'erifiant $|\underline{a}|\leq t\leq n$. On pose $a_{k+1}=t-|\underline{a}|$ et $d_{i}=(a_{1}+\dots a_{i-1})-(a_{i+1}+\dots+a_{k+1}),\;1\leq i\leq k$. Soit $\mathfrak{q}_{n}(t\mid\underline{a}):=\mathfrak{q}^{B}_{n}(t\mid\underline{a}),\;\mathfrak{q}^{C}_{n}(t\mid\underline{a})$ ou $\mathfrak{q}_{n}(t\mid\underline{a}):=\mathfrak{q}^{D}_{n}(t\mid\underline{a})$ si $(t\mid\underline{a})\notin\Xi_{n}$.
 \begin{itemize}
\item[1)] Pour tout $\;1\leq i\leq k$ tel que $d_{i}\neq 0$ et tout $\alpha\in\mathbb{Z}$ tel que $a_{i}+\alpha|d_{i}|\geq 0$, on a 
$$\chi[\mathfrak{q}_{n}(t\mid\underline{a})]=\chi[\mathfrak{q}_{n+\alpha |d_{i}|}(t+\alpha |d_{i}|\mid a_{1},\ldots,a_{i-1},a_{i}+ \alpha |d_{i}|,a_{i+1},\ldots,a_{k})]$$
En particulier, on a 
$$\chi[\mathfrak{q}_{n}(t\mid\underline{a})]=\chi[\mathfrak{q}_{n-a_{i}+a_{i}[|d_{i}|]}(t-a_{i}+a_{i}[|d_{i}|]\mid a_{1},\ldots,a_{i-1},a_{i}[|d_{i}|],a_{i+1},\ldots,a_{k})]$$ 
\item[2)] Pour tout $\;1\leq i\leq k$ tel que $d_{i}=0$, on a 
$$\chi[\mathfrak{q}_{n}(t\mid\underline{a})]=a_{i}+\chi[\mathfrak{q}_{n-a_{i}}(t-a_{i}\mid a_{1},\ldots,a_{i-1},a_{i+1}\ldots,a_{k})]$$
\end{itemize}
\end{theo}

\begin{theo}\label{thm r2}
Soient $\underline{a}=(a_{1},\ldots,a_{k})$ une composition v\'erifiant $1\leq |\underline{a}|=n-1$ $(i.e \;(n\mid\underline{a})\in\Xi_{n})$. On pose $\underline{a}^{'}=(a^{'}_{1},\ldots,a^{'}_{k})=(a_{1},\ldots,a_{k-1},a_{k}+1)$, $d_{k}=-(a^{'}_{1}+\cdots+a^{'}_{k-1})$ et
$d_{i}=(a^{'}_{1}+\cdots+a^{'}_{i-1})-(a^{'}_{i+1}+\cdots+a^{'}_{k})$, $ 1\leq i\leq k-1$.
\begin{itemize}
\item[1)] Pour tout $1\leq i\leq k$ tel que $d_{i}\neq 0$ et tout $\alpha\in\mathbb{Z}$ tel que $a^{'}_{i}+\alpha|d_{i}|\geq 0$, on a 
$$\chi[\mathfrak{q}^{D}_{n}(n\mid a_{1},\ldots,a_{k})]=\Psi[\mathfrak{q}^{A}(n+\alpha|d_{i}|\mid a^{'}_{1},\ldots,a^{'}_{i-1},a^{'}_{i}+\alpha|d_{i}|,a^{'}_{i+1},\ldots,a^{'}_{k})]$$
En particulier, si on pose $t_{i}=a^{'}_{i}-a^{'}_{i}[|d_{i}|]$, on a 
$$\chi[\mathfrak{q}^{D}_{n}(n\mid a_{1},\ldots,a_{k})]=\Psi[\mathfrak{q}^{A}(n-t_{i}\mid a^{'}_{1},\ldots,a^{'}_{i-1},a^{'}_{i}[|d_{i}|],a^{'}_{i+1}\ldots,a^{'}_{k})]$$
\item[2)] Pour tout $1\leq i\leq k$ tel que $d_{i}=0$, on a 
$$\chi[\mathfrak{q}^{D}_{n}(n\mid a_{1},\ldots,a_{k})]=a_{i}+\Psi[\mathfrak{q}^{A}(n-a_{i}\mid a_{1},\ldots,a_{i-1},a_{i+1}\ldots,a_{k})]$$
\end{itemize}
\end{theo}

 Comme cons\'equence de ces deux th\'eor\`emes, nous donnons de nouvelles familles de sous-alg\`ebres biparaboliques de $\mathfrak{sp}(2n)$ (resp. $\mathfrak{so}(2n+1)$, $\mathfrak{so}(2n)$) qui sont de Frobenius (voir lemme \ref{lem F} et th\'eor\`eme \ref{thm a}). Enfin, nous montrons que si $\underline{a}=(a_{1},\ldots,a_{m})$ et $\underline{b}=(b_{1},\ldots,b_{t})$ sont deux compositions de $n$ telles que $\mathfrak{q}_{s}^{A}(\underline{a}\mid\underline{b}):=\mathfrak{q}^{A}(\underline{a}\mid\underline{b})\cap\mathfrak{sl}(n)$ est une sous-alg\`ebre de Frobenius de $\mathfrak{sl}(n)$ (i.e, $\chi[\mathfrak{q}^{A}(\underline{a}\mid\underline{b})]=1$), alors les alg\`ebres $\mathfrak{q}^{D}_{2n}(2a_{1},\ldots,2a_{m}\mid 2b_{1},\ldots,2b_{t-1},2b_{t}-1)$ et $\mathfrak{q}^{D}_{2n}(2a_{1},\ldots,2a_{m-1},2a_{m}-1\mid 2b_{1},\ldots,2b_{t})$
sont deux sous-alg\`ebres de Frobenius de $\mathfrak{so}(4n)$, et toutes les sous-alg\`ebres de Frobenius de $\mathfrak{so}(2n)$ qui sont de la forme $\mathfrak{q}^{D}_{n}(\underline{a}\mid\underline{b})$ o\`u $(\underline{a}\mid\underline{b})\in\Xi_{n}$ sont ainsi obtenues (voir th\'eor\`eme \ref{th B}). En particulier, pour tout $n\geq 1$ et pour toute paire de compositions $(\underline{a}\mid\underline{b})\in\Xi_{2n+1}$, $\mathfrak{q}^{D}_{2n+1}(\underline{a}\mid\underline{b})$ n'est pas une sous-alg\`ebre de Frobenius.\\

Je remercie les professeurs Pierre Torasso et Mohamed Salah khalgui pour d'utiles conversations qui m'ont aid\'e dans ce travail. 

\section{Rappels}

Soit $\mathbold{G}$ un groupe de Lie alg\'ebrique complexe, $\mathfrak{g}$ son alg\`ebre
de Lie et $\mathfrak{g}^{\ast}$ le dual de $\mathfrak{g}$. Au moyen de la repr\'{e}sentation coadjointe, $\mathfrak{g}$ et $\mathbold{G}$ op\`{e}rent dans
$\mathfrak{g}^{*}$ par :
$$(x.f)(y)=f([y,x]), \;\; \forall x,y\in \mathfrak{g}\;\text{et}\;f\in
\mathfrak{g}^{*}$$
$$(x.f)(y)=f(\Ad x^{-1}y),\;\; \forall x\in\mathbold{G},y\in \mathfrak{g}\;\text{et}\;f\in
\mathfrak{g}^{*} $$ 
Pour $f\in \mathfrak{g}^{*}$, soit $\mathbold{G}_{f}$ le stabilisateur de $f$ pour cette action et $\mathfrak{g}_{f}$ son alg\`ebre de Lie:
$$\mathbold{G}_{f}=\{x\in \mathbold{G};f(\Ad x^{-1}y)=f(y),\;\; \forall y\in
\mathfrak{g}\}$$ 
  
$$\mathfrak{g}_{f}=\{x\in \mathfrak{g};f([x,y])=0,\;\; \forall y\in
\mathfrak{g}\}$$ 
On appelle indice de $\mathfrak{g}$ l'entier
$\chi[\mathfrak{g}]$ d\'{e}fini par:
$$\chi[\mathfrak{g}]=\min\{\dim \mathfrak{g}_{f}\;,\;f\in
\mathfrak{g}^{*}\}$$
Si $\chi[\mathfrak{g}]=0$, $\mathfrak{g}$ est dite une alg\`ebre de Frobenius.$\\$

Supposons $\mathfrak{g}$ une alg\`ebre de Lie semi-simple. Soit $\mathfrak{h}$ une sous-alg\`ebre de Cartan de $\mathfrak{g}$, $\Delta\subset\mathfrak{h}^{*}$ le syst\`eme de racines de $\mathfrak{g}$ relativement \`a $\mathfrak{h}$, $\pi:=\lbrace \alpha_{1},\ldots,\alpha_{n}\rbrace$ une base de racines simples num\'erot\'ee, dans le cas o\`u $\mathfrak{g}$ est simple, conform\'ement \`a Bourbaki \cite{Bourbaki}. Alors $\Delta=\Delta^{+}\cup\Delta^{-}$ o\`u $\Delta^{+}$  est l'ensemble des racines positives relativement \`a $\pi$ et $\Delta^{-}=-\Delta^{+}$. Pour $\alpha\in\Delta$, soit $\mathfrak{g}_{\alpha}:=\lbrace x\in\mathfrak{g};\;[h,x]=\alpha(h)x,\;h\in\mathfrak{h}\rbrace$, $\mathfrak{g}_{\alpha}$ est de dimension un.\\
Pour toute partie $\pi^{'}\subset\pi$, soient $\Delta^{+}_{\pi^{'}}=\Delta^{+}\cap\mathbb{N}\pi^{'}$ o\`u $\mathbb{N}\pi^{'}$ d\'esigne l'ensemble des combinaisons lin\'eaires des \'el\'ements de $\pi^{'}$ \`a coefficients dans $\mathbb{N}$, $\Delta^{-}_{\pi^{'}}=-\Delta^{+}_{\pi^{'}}$ 
et $\mathfrak{n}^{\pm}_{\pi^{'}}=\oplus_{\alpha\in\Delta^{\pm}_{\pi^{'}}}\mathfrak{g}_{\alpha}$.\\
\begin{defis}
Soit $(\pi^{'},\pi^{''})$ une paire de parties de $\pi$, la sous-alg\`ebre $\mathfrak{q}_{\pi^{'},\pi^{''}}:=\mathfrak{n}^{+}_{{\pi}^{'}}\oplus \mathfrak{h} \oplus\mathfrak{n}^{-}_{\pi^{''}}$  
est appel\'ee sous-alg\`ebre biparabolique standard de $\mathfrak{g}$.\\
On appelle sous-alg\`ebre biparabolique de $\mathfrak{g}$ toute sous-alg\`ebre de $\mathfrak{g}$ conjugu\'ee par $\mathbold{G}$ \`a une sous-alg\`ebre biparabolique standard de $\mathfrak{g}$.\\
 Si $\pi^{'}=\pi$ ou $\pi^{''}=\pi$,  toute sous-alg\`ebre de $\mathfrak{g}$ conjugu\'ee par $\mathbold{G}$ \`a $\mathfrak{q}_{\pi^{'},\pi^{''}}$ est dite sous-alg\`ebre parabolique de $\mathfrak{g}$.
\end{defis}

Soit $\pi^{'}=\{\alpha_{i_{1}},\ldots,\alpha_{i_{k}}\}$ une partie de $\pi$, on pose $\mathcal{S}_{\pi^{'}}:=(i_{1},i_{2}-i_{1},\ldots,i_{k}-i_{k-1},n+1-i_{k})$, $\mathcal{T}_{\pi^{'}}:=(i_{1},i_{2}-i_{1},\ldots,i_{k}-i_{k-1})$ et on convient que $\mathcal{S}_{\varnothing}=(n+1)$ et $\mathcal{T}_{\varnothing}=\varnothing$. Alors $\mathcal{S}_{\pi^{'}}$ est une composition de $n+1$, $\mathcal{T}_{\pi^{'}}$ est une composition d'un entier $t\leq n$.

 Supposons $\mathfrak{g}=\mathfrak{sl}(n+1)$, on associe \`a toute paire $(\underline{a},\underline{b})$ de compositions  de $n+1$ la sous-alg\`ebre biparabolique $\mathfrak{q}^{A}_{s}(\underline{a}\mid\underline{b}):=\mathfrak{q}_{\pi^{'},\pi^{''}}$ telle que $\underline{a}=\mathcal{S}_{\pi\backslash\pi^{'}}$ et $\underline{b}=\mathcal{S}_{\pi\backslash\pi^{''}}$. La sous-alg\`ebre $\mathfrak{q}^{A}(\underline{a}\mid\underline{b}):=\mathfrak{q}^{A}_{s}(\underline{a}\mid\underline{b})\oplus\mathbb{C}I_{n+1}$, $I_{n+1}$ \'etant la matrice identit\'e d'ordre $n+1$, est une sous-alg\`ebre biparabolique de $\mathfrak{gl}(n+1)$ qui v\'erifie $\chi[\mathfrak{q}^{A}(\underline{a}\mid\underline{b})]=\chi[\mathfrak{q}^{A}_{s}(\underline{a}\mid\underline{b})]+1$. Toutes les sous-alg\`ebres biparaboliques de $\mathfrak{gl}(n+1)$ sont ainsi obtenues (\`a conjugaison pr\`es par le groupe adjoint connexe de $\mathfrak{gl}(n+1)$). La sous-alg\`ebre $\mathfrak{q}^{A}(\underline{a}\mid\underline{b})$ est une sous-alg\`ebre parabolique de $\mathfrak{gl}(n+1)$ si et seulement si $\underline{a}=(n+1)$ ou $\underline{b}=(n+1)$. Soit $(e_{1},\ldots,e_{n})$ la
base canonique de $\mathbb{C}^{n}$, $\cur{V}=\{V_{0}=\{0\}\subsetneq
V_{1}\subsetneq\cdots\subsetneq V_{m}=\mathbb{C}^{n} \}$ et
$\cur{W}=\{\mathbb{C}^{n}=W_{0}\supsetneq W_{1}\supsetneq\cdots\supsetneq
W_{t}=\{0\}\}$ les drapeaux de sous-espaces tels que $V_{i}=\langle
e_{1},\ldots,e_{a_{1}+\cdots+a_{i}}\rangle$, $1\leq i\leq m$, et
$W_{i}=\langle e_{b_{1}+\cdots+b_{i}+1},\ldots,e_{n}\rangle$, $1\leq
i\leq t-1$. Alors $\mathfrak{q}^{A}(\underline{a}\vert\underline{b})$
est le stabilisateur dans $\mathfrak{gl}(n)$ de la paire de drapeaux
$(\cur{V},\cur{W})$.

Supposons $\mathfrak{g}=\mathfrak{sp}(2n)$ (resp. $\mathfrak{so}(2n+1)$, $\mathfrak{so}(2n)$), on associe \`a toute paire de compositions $(\underline{a},\underline{b})$ v\'erifiant $|\underline{a}|\leq n$ et $|\underline{b}|\leq n$, la sous-alg\`ebre biparabolique $\mathfrak{q}^{I}_{n}(\underline{a}\mid\underline{b}):=\mathfrak{q}_{\pi^{'},\pi^{''}},\;I=C$ (resp. $B,\;D$) de $\mathfrak{sp}(2n)$ (resp. $\mathfrak{so}(2n+1)$, $\mathfrak{so}(2n)$) telle que $\underline{a}=\mathcal{T}_{\pi\backslash\pi^{'}}$ et $\underline{b}=\mathcal{T}_{\pi\backslash\pi^{''}}$. \`A conjugaison pr\`es par le groupe adjoint connexe de $\mathfrak{sp}(2n)$ (resp. $\mathfrak{so}(2n+1)$, $\mathfrak{so}(2n)$), toutes les sous-alg\`ebres biparaboliques de $\mathfrak{sp}(2n)$ (resp. $\mathfrak{so}(2n+1)$, $\mathfrak{so}(2n)$) sont ainsi obtenues. La sous-alg\`ebre $\mathfrak{q}^{I}_{n}(\underline{a}\mid\underline{b}),\;I=C$ (resp. $B,\;D$) est une sous-alg\`ebre parabolique de $\mathfrak{sp}(2n)$ (resp. $\mathfrak{so}(2n+1)$, $\mathfrak{so}(2n)$) si et seulement si $\underline{a}=\varnothing$ ou $\underline{b}=\varnothing$. Soit encore
$(e_{1},\ldots,e_{n})$ la base canonique de $\mathbb{C}^{n}$. On munit
$\mathbb{C}^{n}$ de la forme bilin\'eaire $\langle\,,\rangle$ antisym\'etrique (cas
$\mathfrak{sp}(n)$ avec $n$ pair) (resp. sym\'etrique (cas
$\mathfrak{so}(n)$)) telle que $\langle
e_{i},e_{n+1-j}\rangle=\delta_{i,j}$ $1\leq i,j\leq n$, $i+j\leq n+1$.  Soit
$(\underline{a},\underline{b})$ une paire de compositions v\'erifiant
$\vert\underline{a}\vert\leq[\frac{n}{2}]$ et
$\vert\underline{b}\vert\leq[\frac{n}{2}]$,
$\cur{V}=\{V_{0}=\{0\}\subsetneq V_{1}\subsetneq\cdots\subsetneq V_{m} \}$ et
$\cur{W}=\{W_{0}\supsetneq W_{1}\supsetneq\cdots\supsetneq
W_{t}=\{0\}\}$ les drapeaux de sous-espaces isotropes tels que $V_{i}=\langle
e_{1},\ldots,e_{a_{1}+\cdots+a_{i}}\rangle$, $1\leq i\leq m$, et
$W_{i}=\langle e_{n-(b_{1}+\cdots+b_{t-i})+1},\ldots,e_{n}\rangle$, $0\leq
i\leq t-1$. Alors
$\mathfrak{q}_{[\frac{n}{2}]}^{C}(\underline{a}\vert\underline{b})$
(resp. $\mathfrak{q}_{[\frac{n}{2}]}^{I}(\underline{a}\vert\underline{b})$,
$I=B$, si $n$ est impair et $I=D$, si $n$ est pair) est le
stabilisateur dans $\mathfrak{sp}(n)$ (resp. $\mathfrak{so}(n)$) de la
paire de drapeaux $(\cur{V},\cur{W})$.\\

 Soit $\underline{a}=(a_{1},\dots, a_{k})$ une composition d'un entier $n\in\mathbb{N}^{\times}$, on pose $I_{i}=[a_{1}+\dots+a_{i-1}+1,\ldots,a_{1}+\dots+a_{i-1}+a_{i}]\cap\mathbb{N},\;1\leq i\leq k$ et on associe \`a la composition $\underline{a}$, l'involution $\theta_{\underline{a}}$ de $\{1,\ldots,n\}$, d\'efinie par $\theta_{\underline{a}}(x)=2(a_{1}+\dots+a_{i-1})+a_{i}-x+1$, $x\in I_{i},1\leq i\leq k\;$.\\
 
 Soit $(\underline{a},\underline{b})$ une paire de compositions d'un entier $n\in\mathbb{N}^{\times}$, On associe \`a la sous-alg\`ebre biparabolique $\mathfrak{q}^{A}(\underline{a}\mid \underline{b})$ de $\mathfrak{gl}(n)$ un graphe not\'e $\Gamma^{A}(\underline{a}\mid\underline{b})$ et appel\'e m\'eandre de $\mathfrak{q}^{A}(\underline{a}\mid \underline{b})$, dont les sommets sont $n$ points cons\'{e}cutifs situ\'es sur une droite  horizontale D et num\'erot\'es $1,2,\ldots,n$. Il est construit de la mani\`ere suivante: on relie par un arc au dessous (resp. au dessus) de la droite D toute paire de sommets distincts de $\Gamma(\underline{a}\mid\underline{b})$ de la forme $(x,\theta_{\underline{a}}(x))$ (resp.$(x,\theta_{\underline{b}}(x))$), $x\in\{1,\ldots,n\}$. Une composante connexe de $\Gamma^{A}(\underline{a}\mid\underline{b})$ est soit un cycle, soit un segment (voir\cite{MB}).
 $\\$
\begin{Ex}\label{Ex6.1}
 $\Gamma^{A}(2,4,3\mid5,2,2)$=
{\setlength{\unitlength}{0.021in}
\raisebox{-12\unitlength}{%
\begin{picture}(100,20)(-1,-11)
\multiput(0,3)(10,0){9}{\circle*{2}}
\put(20,4){\oval(40,20)[t]}
\put(20,4){\oval(20,10)[t]}
\put(55,4){\oval(10,5)[t]}
\put(75,4){\oval(10,5)[t]}
\put(5,1){\oval(10,5)[b]}
\put(35,1){\oval(30,15)[b]}
\put(35,1){\oval(10,5)[b]}
\put(70,1){\oval(20,10)[b]}
\end{picture}
} }
\end{Ex}

\begin{theo}\label{thm1}\cite{D.K} Soient $\mathfrak{q}^{A}(\underline{a}\mid\underline{b})$ une sous-alg\`ebre biparabolique de $\mathfrak{gl}(n)$ et $\Gamma^{A}(\underline{a}\mid\underline{b})$ son m\'eandre, on a : 
$$\chi[\mathfrak{q}^{A}(\underline{a}\mid\underline{b})]=2\times (nombre\;de\;cycles)+nombre\;de\;segments$$
\end{theo}

\begin{lem}\cite{MB1}\label{lem0}
Soient $\underline{a}=(a_{1},\dots, a_{k})$ et $\underline{b}=(b_{1},\ldots,b_{t})$ deux compositions telles que $|\underline{b}|=|\underline{a}|=n$. On pose $\underline{a}^{-1}=(a_{k},\dots, a_{1})$. On a :
$$\chi[\mathfrak{q}^{A}(\underline{a}\mid\underline{b})]=\chi[\mathfrak{q}^{A}(2n\mid\underline{a}^{-1},\underline{b})]$$
\end{lem}

\begin{theo}\cite{MB1}\label{ind}
Soit $a,b,c,d\;et\;n\in\mathbb{N}^{\times}$ tels que $a+b=c+d=n$, on a 
\begin{itemize}
\item[1)]$\chi[\mathfrak{q}^{A}(a,b\mid n)]=a\wedge b$
\item[2)]$\chi[\mathfrak{q}^{A}(a,b\mid c,d)]=\chi[\mathfrak{q}^{A}(a,b,c\mid n+c)]=(a+b)\wedge (b+c)$
\end{itemize}

\end{theo}

\begin{theo}\cite{MB1}\label{thm0}
Soit $\mathfrak{p}^{A}(a_{1},\ldots,a_{k}):=\mathfrak{q}^{A}(a_{1},\ldots,a_{k}\mid n)$ une sous-alg\`ebre
parabolique de $\mathfrak{gl}(n)$. On pose $d_{k}=-(a_{1}+\cdots+a_{k-1})$ et
$d_{i}=(a_{1}+\cdots+a_{i-1})-(a_{i+1}+\cdots+a_{k})$, $ 1\leq i\leq
k-1$.
\begin{itemize}
\item[1)] Pour tout $ 1\leq i\leq k$ tel que $d_{i}\neq 0$ et tout $\alpha\in\mathbb{Z}$ tel que $a_{i}+\alpha|d_{i}|\geq 0$, on a 
$$\chi[\mathfrak{p}^{A}(a_{1},\ldots,a_{k})]=\chi[\mathfrak{p}^{A}(a_{1},\ldots,a_{i}+\alpha|d_{i}|,\ldots,a_{k})]$$
En particulier, on a 
$$\chi[\mathfrak{p}^{A}(a_{1},\ldots,a_{k})]=\chi[\mathfrak{p}^{A}(a_{1},\ldots,a_{i-1},a_{i}[|d_{i}|],a_{i+1}\ldots,a_{k})]$$
\item[2)] Pour tout $ 1\leq i\leq k$ tel que $d_{i}=0$, on a 
$$\chi[\mathfrak{p}^{A}(a_{1},\ldots,a_{k})]=a_{i}+\chi[\mathfrak{p}^{A}(a_{1},\ldots,a_{i-1},a_{i+1}\ldots,a_{k})]$$
\end{itemize}
\end{theo}

\section{sous-alg\`ebres biparaboliques de $sp(2n)$}

Soient $n\in\mathbb{N}^{\times}$ et $(\underline{a},\underline{b})$ une paire de compositions v\'erifiant $|\underline{a}|\leq n$ et $|\underline{b}|\leq n$. On associe \`a la sous-alg\`ebre biparabolique $\mathfrak{q}^{C}_{n}(\underline{a}\mid \underline{b})$ de $\mathfrak{sp}(2n)$ le m\'eandre $\Gamma^{C}_{n}(\underline{a}\mid\underline{b}):=\Gamma^{A}(\underline{a}^{'}\mid\underline{b}^{'})$, o\`u $\underline{a}^{'}=(a_{1},\ldots,a_{k},2(n-|\underline{a}|),a_{k},\ldots,a_{1})$ et $\underline{b}^{'}=(b_{1},\ldots,b_{t},2(n-|\underline{b}|),b_{t},\ldots,b_{1})$. $\underline{a}^{'}$ et $\underline{b}^{'}$ sont deux compositions de $2n$, le nombre de sommets de $\Gamma^{C}_{n}(\underline{a}\mid\underline{b})$ est alors \'egal \`a $2n$. Soit $\sigma$ la sym\'etrie par rapport \` a la droite verticale passant par le milieu des sommets $n$ et $n+1$. Par construction, le m\'eandre $\Gamma^{C}_{n}(\underline{a}\mid\underline{b})$ est invariant par $\sigma$. Un cycle (resp. segment) X est dit invariant si $\sigma(X)=X$.
\begin{Ex}
\medskip$\\$
$\Gamma^{C}_{8}(2,5\mid 1,4)={\setlength{\unitlength}{0.021in}
\raisebox{-12\unitlength}{%
\begin{picture}(100,20)(-10,-10)
\multiput(0,3)(15,0){16}{\circle*{2}}
\put(37,5){\oval(45,20)[t]}
\put(37,5){\oval(15,10)[t]}
\put(112,5){\oval(75,30)[t]}
\put(112,5){\oval(45,20)[t]}
\put(112,5){\oval(15,10)[t]}
\put(187,5){\oval(45,20)[t]}
\put(187,5){\oval(15,10)[t]}

\put(7,1){\oval(15,10)[b]}
\put(60,1){\oval(60,20)[b]}
\put(60,1){\oval(30,10)[b]}
\put(112,1){\oval(15,10)[b]}
\put(164,1){\oval(60,20)[b]}
\put(164,1){\oval(30,10)[b]}
\put(217,1){\oval(15,10)[b]}
\end{picture}
} }
\put(30,-25){\line(0,1){60}}
$
 \end{Ex}
 
\begin{theo}\label{thm2}\cite{meander.C} Soient $\mathfrak{q}^{C}_{n}(\underline{a}\mid\underline{b})$ une sous-alg\`ebre biparabolique de $\mathfrak{sp}(2n)$ et $\Gamma^{C}_{n}(\underline{a}\mid\underline{b})$ son m\'eandre, ona : 
$$\chi[\mathfrak{q}^{C}_{n}(\underline{a}\mid\underline{b})]=nombre\;de\;cycles\;+\frac{1}{2}\;(nombre\;de\;segments\;non\;invariants)$$
\end{theo}

\begin{cor}\label{cor1} 
Soient $\underline{a}=(a_{1},\dots, a_{k})$ et $\underline{b}=(b_{1},\ldots,b_{t})$ deux compositions telles que $|\underline{a}|\leq n$ et $|\underline{b}|\leq n$. 
\begin{itemize}
\item[1)]$\chi[\mathfrak{q}^{C}_{n}(\underline{a}\mid\underline{b})]=\chi[\mathfrak{q}^{C}_{n}(\underline{b}\mid\underline{a})]$
\item[2)]$\chi[\mathfrak{q}^{C}_{n}(\underline{a}\mid\underline{b})]=\chi[\mathfrak{q}^{C}_{\max(|\underline{a}|,|\underline{b}|)}(\underline{a}\mid\underline{b})]+n-\max(|\underline{a}|,|\underline{b}|)$
\item[3)]$\chi[\mathfrak{q}^{C}_{n}(\underline{a}\mid\varnothing)]=\sum_{1\leq i\leq k}[\frac{a_{i}}{2}]+(n-|\underline{a}|)$
\item[4)] S'il existe $1\leq i\leq  k$ et $1\leq j\leq  t$ tels que $a_{1}+\dots+a_{i}=b_{1}+\dots+b_{j}$, on a : $\chi[\mathfrak{q}^{C}_{n}(\underline{a}\mid\underline{b})]=\chi[\mathfrak{q}^{A}(a_{1},\ldots,a_{i}\mid b_{1},\ldots,b_{j})]+\chi[\mathfrak{q}^{C}_{n-(a_{1}+\dots+a_{i})}(a_{i+1},\ldots,a_{k}\mid b_{j+1},\ldots,b_{t})]$
o\`u $\mathfrak{q}^{A}(a_{1},\ldots,a_{i}\mid b_{1},\ldots,b_{j})$ est la sous-alg\`ebre biparabolique de $\mathfrak{gl}(a_{1}+\dots+a_{i})$ associ\'ee \`a la paire de compositions $((a_{1},\ldots,a_{i}),(b_{1},\ldots,b_{j}))$. En particulier, si $|\underline{a}|=|\underline{b}|$, alors \\
$\chi[\mathfrak{q}^{C}_{n}(\underline{a}\mid\underline{b})]=\chi[\mathfrak{q}^{A}(\underline{a}\mid\underline{b})]+\chi[\mathfrak{q}^{C}_{n-|\underline{a}|}(\varnothing\mid\varnothing)]=\chi[\mathfrak{q}^{A}(\underline{a}\mid\underline{b})]+n-|\underline{a}|$
 
\end{itemize}
\end{cor}

\begin{lem}\label{lem0'} Soient $\underline{a}=(a_{1},\dots, a_{k})$ et $\underline{b}=(b_{1},\ldots,b_{t})$ deux compositions telles que $|\underline{b}|\leq|\underline{a}|\leq n$. On pose $\underline{a}^{-1}=(a_{k},\dots, a_{1})$. On a :
$$\chi[\mathfrak{q}^{C}_{n}(\underline{a}\mid\underline{b})]=\chi[\mathfrak{q}^{C}_{n+|\underline{a}|}(2|\underline{a}|\mid\underline{a}^{-1},\underline{b})]$$

\end{lem}

\begin{proof}
D'apr\`es le corollaire \ref{cor1}, on peut supposer $|\underline{a}|=n$. Soit $I=[1,n]$, $I^{'}=[1,|\underline{b}|]$ et $\underline{c}=(a_{k},\dots, a_{1},b_{1},\ldots,b_{t})$. On v\'erifie que $\theta_{a}$ (resp. $\theta_{b}$) est la restriction de $\theta_{n}\theta_{c}\theta_{n}$ (resp. $\theta_{n+|\underline{b}|}\theta_{c}\theta_{n+|\underline{b}|}$) \`a $I$ (resp. $I^{'}$). Puisque les m\'eandres $\Gamma^{C}_{2n}(2n\mid\underline{a}^{-1},\underline{b})$ et $\Gamma^{C}_{n}(\underline{a}\mid\underline{b})$ sont invariants par la sym\'etrie $\sigma$ d\'efinie au d\'ebut de ce paragraphe, il existe alors une bijection entre les ensembles des composantes connexes des deux m\'eandres conservant le nombre des cycles et le nombre des segments non invariants (voir exemple \ref{Ex1}). Le r\'esultat se d\'eduit alors du th\'eor\`eme \ref{thm2}.

\end{proof}

\begin{Ex}\label{Ex1}
 Consid\'erons la paire de compositions $((2,3),(3,1))$, les arcs en bleu dans le m\'eandre $\Gamma^{C}_{5}(2,3\mid 3,1)$ sont envoy\'es sur les arcs en bleu dans le m\'eandre $\Gamma^{C}_{10}(10\mid 3,2,3,1)$. 

\medskip$\\$
$\Gamma^{C}_{5}(2,3\mid 3,1)=
{\setlength{\unitlength}{0.021in}
\raisebox{-12\unitlength}{%
\begin{picture}(49.5,-10)(-10,-10)
\multiput(63,3)(10,0){10}{\circle*{2}}
\put(73,4){\oval(20,15)[t]}
\put(143,4){\oval(20,15)[t]}
\put(108,4){\oval(10,10)[t]}

\textcolor{blue}{\put(67.5,3){\oval(10,10)[b]}}
\textcolor{blue}{\put(145,3){\oval(10,10)[b]}}
\textcolor{blue}{\put(88,3){\oval(20,15)[b]}}
\textcolor{blue}{\put(115,3){\oval(20,15)[b]}}
\end{picture}
} }
\put(100,-50){\line(0,1){95}}$ $\\$
  
\medskip$\\$
$\Gamma^{C}_{10}(10\mid 3,2,3,1)=
{\setlength{\unitlength}{0.021in}
\raisebox{-12\unitlength}{%
\begin{picture}(50,-10)(-12,-10)
\multiput(0,3)(10,0){20}{\circle*{2}}

\put(60,4){\oval(20,15)[t]}
\put(130,4){\oval(20,15)[t]}
\put(95,4){\oval(10,10)[t]}

\textcolor{blue}{\put(10,4){\oval(20,15)[t]}}
\textcolor{blue}{\put(32,4){\oval(10,10)[t]}}

\textcolor{blue}{\put(175,4){\oval(20,15)[t]}}
\textcolor{blue}{\put(147,4){\oval(10,10)[t]}}

\put(37,3){\oval(90,50)[b]}
\put(37,3){\oval(70,40)[b]}
\put(37,3){\oval(50,30)[b]}
\put(37,3){\oval(30,20)[b]}
\put(37,3){\oval(10,10)[b]}

\put(137,3){\oval(90,50)[b]}
\put(137,3){\oval(70,40)[b]}
\put(137,3){\oval(50,30)[b]}
\put(137,3){\oval(30,20)[b]}
\put(137,3){\oval(10,10)[b]}

\end{picture}
} }
\put(82.25,-50){\line(0,1){95}}$ $\\$
\end{Ex}

\begin{rem}
Compte tenu du corollaire \ref{cor1} et du lemme \ref{lem0'}, pour \'etudier l'indice des sous-alg\`ebres biparaboliques de $\mathfrak{sp}(2n)$, on peut se ramener au cas des sous-alg\`ebres biparaboliques de la forme $\mathfrak{q}^{C}_{n}(n\mid\underline{a})$, o\`u $\underline{a}$ est une composition d'un entier inf\'erieur ou \'egal \`a $n$. 
\end{rem}

\begin{lem} Soient $\underline{a}=(a_{1},\ldots,a_{k})$ une composition v\'erifiant $|\underline{a}|\leq n$ et $s=n-|\underline{a}|$. Alors, pour tout $t\in\mathbb{N}$, on a
$$\chi[\mathfrak{q}^{C}_{n+4ts}(n+4ts\mid \underbrace{2s,\ldots,2s}_{t},\underline{a},\underbrace{2s,\ldots,2s}_{t})]=\chi[\mathfrak{q}^{C}_{n}(n\mid\underline{a})]$$
En particulier, si $\mathfrak{q}^{C}_{n}(n\mid\underline{a})$ est une sous-alg\`ebre de Frobenius de $\mathfrak{sp}(2n)$, alors, pour tout $t\in\mathbb{N}$, la sous-alg\`ebre biparabolique 
$\mathfrak{q}^{C}_{n+4ts}(n+4ts\mid \underbrace{2s,\ldots,2s}_{t},\underline{a},\underbrace{2s,\ldots,2s}_{t})$ est une sous-alg\`ebre de Frobenius de $\mathfrak{sp}(2(n+4ts))$.
\end{lem}

\begin{proof}
La figure suivante montre que le r\'esultat est vrai pour $t=1$.
\medskip$\\$
$\Gamma^{C}_{n+4s}(n+4s\mid 2s,\underline{a},2s)=
{\setlength{\unitlength}{0.021in}
\raisebox{-12\unitlength}{%
\begin{picture}(50,-10)(-17,-10)
\multiput(-13,3)(10,0){5}{\circle*{2}}
\multiput(50,3)(10,0){1}{\circle*{2}}
\multiput(65,3)(20,0){2}{\circle*{2}}
\multiput(95,3)(10,0){6}{\circle*{2}}

\textcolor{blue}{\put(2,3){\oval(30,20)[t]}}
\textcolor{blue}{\put(-0.5,3){\oval(10,10)[t]}}
\textcolor{blue}{\put(105,3){\oval(30,20)[t]}}
\textcolor{blue}{\put(102.5,3){\oval(10,10)[t]}}

\put(31.5,3){\oval(23.5,20)[t]}
\put(68,3){\oval(21,20)[t]}

\put(46,3){\ldots}
\put(160,3){\ldots}
\put(170,3){\ldots}

\textcolor{blue}{\put(58.5,3){\oval(158,100)[b]}}
\textcolor{blue}{\put(56,3){\oval(138,85)[b]}}
\textcolor{blue}{\put(53.5,3){\oval(118,70)[b]}}
\textcolor{blue}{\put(52,3){\oval(98,55)[b]}}
\put(51.5,3){\oval(78,40)[b]}

\textcolor{blue}{\put(148,3){\oval(55,20)[lt]}}
\textcolor{blue}{\put(145.5,3){\oval(35,10)[lt]}}

\qbezier[10](50,-3)(50,-9)(50,-15)
\put(10,3){\line(1,0){23}}
\put(47.5,3){\line(1,0){20}}
\put(78,3){\line(1,0){10}}
\put(82,4){$s$}
\put(18,5){$a_{1}$}
\put(55,5){$a_{k}$}
\end{picture}
} }
$\put(180,-50){\line(0,1){95}}$ 
$
$\\$
$\\$
$\\$
$\\$
$\\$
$\\$
Le r\'esultat s'en d\'eduit alors par r\'ecurrence sur $t$. 
\end{proof}

\begin{lem}\label{lemM}
Soient $\underline{a}=(a_{1},\ldots,a_{k})$ une composition v\'erifiant $|\underline{a}|\leq n$. Supposons qu'il existe $1\leq i\leq k$ tel que  $|\underline{a}|_{i}:=a_{1}+\dots+a_{i}\leq n-|\underline{a}|$. Alors, on a
$$\chi[\mathfrak{q}^{C}_{n}(n\mid\underline{a})]=\sum_{1\leq j\leq i}[\frac{a_{j}}{2}]+\chi[\mathfrak{q}^{C}_{n-2|\underline{a}|_{i}}(n-2|\underline{a}|_{i}\mid a_{i+1},\ldots,a_{k})]$$
En particulier, si $|\underline{a}|\leq [\frac{n}{2}]$, on a 
$$\chi[\mathfrak{q}^{C}_{n}(n\mid\underline{a})]=\sum_{1\leq j\leq k}[\frac{a_{j}}{2}]+[\frac{n-2|\underline{a}|}{2}]$$
\end{lem}

\begin{proof}

Le m\'eandre $\Gamma^{C}_{n}(n\mid\underline{a})$ est r\'eunion disjointe des m\'eandres $\Gamma^{C}_{2|\underline{a}|_{i}}(2|\underline{a}|_{i}\mid a_{1},\ldots,a_{i})$ et $\Gamma^{C}_{n-2|\underline{a}|_{i}}(n-2|\underline{a}|_{i}\mid a_{i+1},\ldots,a_{k})$. D'autre part, on v\'erifie que le m\'eandre $\Gamma^{C}_{2|\underline{a}|_{i}}(2|\underline{a}|_{i}\mid a_{1},\ldots,a_{i})$ est compos\'e de $\sum_{1\leq j\leq i}[\frac{a_{j}}{2}]$ cycles et $\sum_{1\leq j\leq i}[\frac{a_{j}+1}{2}]-\sum_{1\leq j\leq i}[\frac{a_{j}}{2}]$ segments qui sont tous invariants. Le r\'esultat se d\'eduit alors du th\'eor\`eme \ref{thm2}.
\end{proof}

\begin{lem}\label{lem1}
Soient $\underline{a}=(a_{1},\ldots,a_{k})$ une composition v\'erifiant $|\underline{a}|\leq n$. On pose $a_{k+1}=n-|\underline{a}|$. Alors, 
\begin{itemize}
\item[1)]Soient $a_{i,j}=(a_{1}+\dots +a_{i})-(a_{j}+\dots+a_{k+1})$ et $a^{i,j}=(a_{i+1}+\dots +a_{j-1})+|a_{i,j}|$, $1\leq i<j\leq k+1$. On a

$\chi[\mathfrak{q}^{C}_{n}(n\mid \underline{a})]=\begin{cases}\chi[\mathfrak{q}^{C}_{n+a^{i,j}}(n+a^{i,j}\mid a_{1},\ldots,a_{i},a^{i,j},a_{i+1},\ldots,a_{k})]\;si\;a_{i,j}< 0\\
\chi[\mathfrak{q}^{C}_{n+a^{i,j}}(n+a^{i,j}\mid a_{1},\ldots,a_{j-1},a^{i,j},a_{j},\ldots,a_{k})]\;si\;a_{i,j}\geq 0\\
\end{cases}$\\
\item[2)] Soit $d_{i}=(a_{1}+\dots +a_{i-1})-(a_{i+1}+\dots+a_{k+1}),\;1\leq i\leq k$. Supposons qu'il existe $1\leq i\leq k$ tel que $a_{i}\geq |d_{i}|$. On a 
\begin{itemize}
\item[i)] $\chi[\mathfrak{q}^{C}_{n}(n\mid \underline{a})]=\begin{cases}\chi[\mathfrak{q}^{C}_{n+a_{i}+d_{i}}(n+a_{i}+d_{i}\mid a_{1},\ldots,a_{i},a_{i}+d_{i},a_{i+1},\ldots,a_{k})]\;si\;d_{i}< 0\\
\chi[\mathfrak{q}^{C}_{n+a_{i}-d_{i}}(n+a_{i}-d_{i}\mid a_{1},\ldots,a_{i-1},a_{i}-d_{i},a_{i},\ldots,a_{k})]\;si\;d_{i}\geq 0\\
\end{cases}$\\
\item[ii)] $\chi[\mathfrak{q}^{C}_{n}(n\mid \underline{a})]=\chi[\mathfrak{q}^{C}_{n-|d_{i}|}(n-|d_{i}|\mid a_{1},\ldots,a_{i-1},a_{i}-|d_{i}|,a_{i+1},\ldots,a_{k})]$.
\end{itemize}
\end{itemize}
\end{lem}

\begin{proof}
\begin{itemize}

\item[1)] Supposons $a_{i,j}\leq 0$. Il suit de la d\'emonstration du lemme \ref{lem0'} qu'il existe une bijection entre les composantes connexes des m\'eandres $\Gamma^{C}_{a^{i,j}}(a^{i,j}\mid a_{i+1},\ldots,a_{j-1})$ et $\Gamma^{C}_{2a^{i,j}}(2a^{i,j}\mid a^{i,j},a_{i+1},\ldots,a_{j-1})$  conservant le nombre de cycles et le nombre de segments invariants. Compte tenu de la sym\'etrie $\sigma$, cette bijection se prolonge de mani\`ere naturelle en une bijection entre les composantes connexes des m\'eandres $\Gamma^{C}_{n}(n\mid \underline{a})$ et $\Gamma^{C}_{n+a^{i,j}}(n+a^{i,j}\mid a_{1},\ldots,a_{i},a^{i,j},a_{i+1},\ldots,a_{k})$ qui conserve le nombre de cycles et le nombre de segments invariants. Le r\'esultat d\'ecoule du th\'eor\`eme \ref{lem0'}. Le cas  $a_{i,j}\geq 0$ se d\'emontre de la m\^eme mani\`ere (voir exemple \ref{Ex2}).
\item[2)]
 \begin{itemize}
\item[i)] Il suffit de remarquer que $\begin{cases} a^{i,i+1}=a_{i,i+1}=a_{i}+d_{i}\;si\;d_{i}< 0\\
a^{i-1,i}=-a_{i-1,i}=a_{i}-d_{i}\;si\;d_{i}\geq 0\\
\end{cases}$ (voir exemple \ref{Ex3}).
\item[ii)] Soit $\underline{b}=(b_{1},\ldots,b_{k})$ tel que $b_{j}=a_{j}$ si $j\neq i$ et $b_{i}=a_{i}-|d_{i}|$. On v\'erifie que $b_{i-1,i+1}=d_{i}$ et $b^{i-1,i+1}=a_{i}$, il suit de 1) que l'on a\\
 $\chi[\mathfrak{q}^{C}_{n}(\underline{b}\mid n)]=\begin{cases}\chi[\mathfrak{q}^{C}_{n+a_{i}+d_{i}}(n+a_{i}+d_{i}\mid a_{1},\ldots,a_{i},a_{i}+d_{i},a_{i+1},\ldots,a_{k})]\;si\;d_{i}< 0\\
\chi[\mathfrak{q}^{C}_{n+a_{i}-d_{i}}(n+a_{i}-d_{i}\mid a_{1},\ldots,a_{i-1},a_{i}-d_{i},a_{i},\ldots,a_{k})]\;si\;d_{i}\geq 0\\
\end{cases}$\\
Le r\'esultat se d\'eduit de i)
\end{itemize}
 \end{itemize}
\end{proof}

\begin{Ex} \label{Ex2}
Consid\'erons la composition $\underline{a}=(a_{1},a_{2})=(3,2)$ et $n=7$, alors $|\underline{a}|=5$, $a_{3}=n-|\underline{a}|=2$, $a_{1,3}=a_{1}-a_{3}=1>0$ et $a^{1,3}=a_{2}+a_{1,3}=3$. Par suite, on a \\
$\chi[ \mathfrak{q}^{C}_{7}(7\mid 3,2)]=\chi[\mathfrak{q}^{C}_{10}(10\mid 3,2,3)]$.

\medskip$\\$
$\Gamma^{C}_{7}(7\mid 3,2)=
{\setlength{\unitlength}{0.021in}
\raisebox{-12\unitlength}{%
\begin{picture}(50,-10)(-18,-10)
\multiput(40,3)(10,0){14}{\circle*{2}}
\put(50,4){\oval(20,15)[t]}
\put(160,4){\oval(20,15)[t]}

\put(105,4){\oval(30,20)[t]}
\put(105,4){\oval(10,10)[t]}

\put(75,4){\oval(10,10)[t]}
\put(135,4){\oval(10,10)[t]}

\put(70,3){\oval(60,35)[b]}
\put(70,3){\oval(40,25)[b]}
\put(70,3){\oval(20,15)[b]}

\put(140,3){\oval(60,35)[b]}
\put(140,3){\oval(40,25)[b]}
\put(140,3){\oval(20,15)[b]}
\end{picture}
} }
\put(107,-50){\line(0,1){95}}$ $\\$
  
\medskip$\\$
$\Gamma^{C}_{10}(10\mid 3,2,3)=
{\setlength{\unitlength}{0.021in}
\raisebox{-12\unitlength}{%
\begin{picture}(50,-10)(-17,-10)
\multiput(0,3)(10,0){20}{\circle*{2}}

\textcolor{blue}{\put(60,4){\oval(20,15)[t]}}
\textcolor{blue}{\put(127,4){\oval(20,15)[t]}}

\put(92,4){\oval(30,20)[t]}
\put(92,4){\oval(10,10)[t]}

\put(32,4){\oval(10,10)[t]}    
\put(152,4){\oval(10,10)[t]}

\put(7,4){\oval(20,15)[t]}
\put(177,4){\oval(20,15)[t]}

\put(42,3){\oval(90,50)[b]}
\put(42,3){\oval(70,40)[b]}
\textcolor{blue}{\put(42,3){\oval(50,30)[b]}}
\textcolor{blue}{\put(39,3){\oval(30,20)[b]}}
\textcolor{blue}{\put(37,3){\oval(10,10)[b]}}

\put(137,3){\oval(90,50)[b]}
\put(137,3){\oval(70,40)[b]}
\textcolor{blue}{\put(137,3){\oval(50,30)[b]}}
\textcolor{blue}{\put(135,3){\oval(30,20)[b]}}
\textcolor{blue}{\put(132,3){\oval(10,10)[b]}}

\textcolor{blue}{\qbezier[25](5,2)(15,-10)(24,2)}
\textcolor{blue}{\qbezier[25](133,2)(143,-10)(152,2)}

\end{picture}
} }
\put(90,-50){\line(0,1){95}}$ 
$\\$
\end{Ex}

\begin{Ex} \label{Ex3}
Consid\'erons la composition $\underline{a}=(a_{1},a_{2})=(3,3)$ et $n=8$, alors $|\underline{a}|=6$, $a_{3}=n-|\underline{a}|=2$, $d_{2}=a_{1}-a_{3}=1>0$ et $a_{2}-d_{2}=2$. On a $\chi[ \mathfrak{q}_{8}(8\mid 3,3)]=\chi[\mathfrak{q}_{10}(10\mid 3,2,3)]$. Les arcs en bleu dans le m\'eandre $\Gamma^{C}_{8}(8\mid 3,3)$ sont envoy\'es sur les arcs en bleu dans le m\'eandre $\Gamma^{C}_{10}(10\mid 3,2,3)$.

\medskip$\\$
$\Gamma^{C}_{8}(8\mid 3,3)=
{\setlength{\unitlength}{0.021in}
\raisebox{-12\unitlength}{%
\begin{picture}(50,-10)(-10,-10)
\multiput(40,3)(10,0){16}{\circle*{2}}
\put(50,4){\oval(20,15)[t]}
\put(180,4){\oval(20,15)[t]}

\put(115,4){\oval(30,20)[t]}
\put(115,4){\oval(10,10)[t]}

\put(80,4){\oval(20,15)[t]}
\put(150,4){\oval(20,15)[t]}

\put(75,3){\oval(70,40)[b]}
\put(75,3){\oval(50,30)[b]}
\put(75,3){\oval(30,20)[b]}
\textcolor{blue}{\put(75,1){\oval(10,10)[b]}}

\put(155,3){\oval(70,40)[b]}
\put(155,3){\oval(50,30)[b]}
\put(155,3){\oval(30,20)[b]}
\textcolor{blue}{\put(155,1){\oval(10,10)[b]}}
\end{picture}
} }
\put(110,-50){\line(0,1){95}}$ $\\$
  
\medskip$\\$
$\Gamma^{C}_{10}(10\mid 3,2,3)=
{\setlength{\unitlength}{0.021in}
\raisebox{-12\unitlength}{%
\begin{picture}(50,-10)(-19,-10)
\multiput(0,3)(10,0){20}{\circle*{2}}

\put(60,4){\oval(20,15)[t]}
\put(130,4){\oval(20,15)[t]}

\put(95,4){\oval(30,20)[t]}
\put(95,4){\oval(10,10)[t]}

\put(10,4){\oval(20,15)[t]}
\textcolor{blue}{\put(35,5){\oval(10,10)[t]}}

\put(180,4){\oval(20,15)[t]}
\textcolor{blue}{\put(155,5){\oval(10,10)[t]}}

\put(45,3){\oval(90,50)[b]}
\put(45,3){\oval(70,40)[b]}
\put(45,3){\oval(50,30)[b]}
\put(45,3){\oval(30,20)[b]}
\put(45,3){\oval(10,10)[b]}

\put(145,3){\oval(90,50)[b]}
\put(145,3){\oval(70,40)[b]}
\put(145,3){\oval(50,30)[b]}
\put(145,3){\oval(30,20)[b]}
\put(145,3){\oval(10,10)[b]}

\end{picture}
} }
\put(92.5,-50){\line(0,1){95}}$ 
$\\$
\end{Ex}

Le th\'eor\`eme suivant est une cons\'equence im\'ediate du corollaire \ref{cor1} et du lemme \ref{lem1}.

\begin{theo}\label{thm0'} Soient $\underline{a}=(a_{1},\ldots,a_{k})$ une composition v\'erifiant $|\underline{a}|\leq n$. On pose $a_{k+1}=n-|\underline{a}|$ et $d_{i}=(a_{1}+\dots +a_{i-1})-(a_{i+1}+\dots+a_{k+1}),\;1\leq i\leq k$.
 \begin{itemize}
\item[1)]Pour tout $1\leq i\leq k$ tel que $d_{i}\neq 0$ et tout $\alpha\in\mathbb{Z}$ tel que $a_{i}+\alpha|d_{i}|\geq 0$, on a 
$$\chi[\mathfrak{q}^{C}_{n}(n\mid\underline{a})]=\chi[\mathfrak{q}^{C}_{n+\alpha |d_{i}|}(n+\alpha |d_{i}|\mid a_{1},\ldots,a_{i-1},a_{i}+\alpha |d_{i}|,a_{i+1},\ldots,a_{k})]$$
En particulier, on a 
$$\chi[\mathfrak{q}^{C}_{n}(n\mid\underline{a})]=\chi[\mathfrak{q}^{C}_{n-a_{i}+a_{i}[|d_{i}|]}(n-a_{i}+a_{i}[|d_{i}|]\mid a_{1},\ldots,a_{i-1},a_{i}[|d_{i}|],a_{i+1},\ldots,a_{k})]$$ 
\item[2)] Pour tout $1\leq i\leq k$ tel que $d_{i}=0$, on a 
$$\chi[\mathfrak{q}^{C}_{n}(n\mid\underline{a})]=a_{i}+\chi[\mathfrak{q}^{C}_{n-a_{i}}(a_{1},\ldots,a_{i-1},a_{i+1}\ldots,a_{k})]$$
\end{itemize} 
\end{theo}

\begin{rem}\label{rem}
Le lemme 2.4 de \cite{MB1} montre que si la composition $\underline{a}=(a_{1},\ldots,a_{k})$ et l'entier $n$ v\'erifient $|\underline{a}|\leq n$, alors il existe $\;1\leq i\leq k$ tel que $a_{i}\geq |d_{i}|$ ou $|\underline{a}|\leq [\frac{n}{2}]$. 
\end{rem}

\begin{rem}
Compte tenu du corollaire \ref{cor1}, des lemmes \ref{lem0'} et \ref{lemM} et de la remarque \ref{rem}, le th\'eor\`eme pr\'ec\'edent donne un algorithme de r\'eduction permettant le calcul de l'indice des sous-alg\`ebres biparaboliques de $\mathfrak{sp}(2n)$. 
\end{rem}

\begin{Ex} Consid\'erons la sous-alg\`ebre biparabolique $\mathfrak{q}^{C}_{200}(15,185\mid 17,61,117)$ de $\mathfrak{sp}(400)$. Compte tenu du lemme \ref{lem0'}, on a 
$$ \chi[\mathfrak{q}^{C}_{200}(15,185\mid 17,61,117)]=\chi[\mathfrak{q}^{C}_{400}(400\mid 185,15,17,61,117)]$$
Puis, en appliquant successivement le th\'eor\`eme \ref{thm0'}, on a 
\begin{align*}
 \chi[\mathfrak{q}^{C}_{400}(400\mid 185,15,17,61,117)]&=\chi[\mathfrak{q}^{C}_{385}(385\mid 185,17,61,117)]\\
 &=\chi[\mathfrak{q}^{C}_{369}(369\mid 185,1,61,117)]\\
 &=\chi[\mathfrak{q}^{C}_{185}(185\mid 1,1,61,117)]\\
 &=\chi[\mathfrak{q}^{C}_{69}(69\mid 1,1,61,1)]\\
 &=\chi[\mathfrak{q}^{C}_{9}(9\mid 1,1,1,1)]\\
 \end{align*}
 Il suit du lemme \ref{lemM} que l'on a 
 \begin{align*}
 \chi[\mathfrak{q}^{C}_{400}(400\mid 185,15,17,61,117)]&=\chi[\mathfrak{q}^{C}_{9}(9\mid 1,1,1,1)]\\
 &=0
 \end{align*}
\end{Ex}

\begin{lem}\label{lem F}

Pour tout $k\in\mathbb{N}^{\times}$ et $(\alpha_{1},\ldots,\alpha_{k})\in\mathbb{N}^{k}$, soit $a_{k+1}=k$ et $\underline{a}=(a_{1},\ldots,a_{k})$ la composition d\'efinie par $a_{i}=1+\alpha_{i}(a_{i+1}+\dots+a_{k+1}-i+1),\;1\leq i\leq k$ et $r=|\underline{a}|+k$. Alors, 
$\mathfrak{q}^{C}_{r}(r\mid\underline{a})$ est une sous-alg\`ebre de Frobenius de $\mathfrak{sp}(2r)$.

\end{lem}

\begin{proof}
 Soit $s_{i}=a_{i+1}+\dots+a_{k+1}-i+1, \;1\leq i\leq k$, alors $a_{i}[s_{i}]=1$. Utilisant la r\'eduction du th\'eor\`eme \ref{thm0'}, on obtient 
 $$\chi[\mathfrak{q}^{C}_{r}(r\mid\underline{a})]=\chi[\mathfrak{q}^{C}_{2k-1}(2k-1\mid \underbrace{1,\ldots,1}_{k-1})]$$
 Il suit du lemme \ref{lemM} que l'on a
  $$\chi[\mathfrak{q}^{C}_{r}(r\mid\underline{a})]=0$$
\end{proof}

\begin{lem}\label{lem3}
Soient $\underline{a}=(a_{1},\ldots,a_{k})$ et $\underline{b}=(b_{1},\ldots,b_{t})$ deux compositions v\'erifiant $|\underline{b}|\leq|\underline{a}|\leq n$. Supposons $a_{1}>b_{1}$, on a :
$$\chi[\mathfrak{q}^{C}_{n}(\underline{a}\mid\underline{b})]=\chi[\mathfrak{q}^{C}_{n-b_{1}}(a_{1}-b_{1}-a_{1}[a_{1}-b_{1}],a_{1}[a_{1}-b_{1}],a_{2},\ldots,a_{k}\mid b_{2}\ldots,b_{t})]$$.
\end{lem}

\begin{proof}
D'apr\`es le corollaire \ref{cor1}, on peut supposer que $|\underline{a}|=n$. Utilisant le lemme \ref{lem0'} et le th\'eor\`eme \ref{thm0'}, on a :\\
\begin{align*}
\chi[\mathfrak{q}^{C}_{n}(\underline{a}\mid\underline{b})]&=\chi[\mathfrak{q}^{C}_{2n}(2n\mid a_{k},\ldots,a_{1},\underline{b})]\\
&=\chi[\mathfrak{q}^{C}_{2n-b_{1}}(2n-b_{1}\mid a_{k},\ldots,a_{1},b_{2},\ldots,b_{t})]\\
&=\chi[\mathfrak{q}^{C}_{2n-b_{1}-a_{1}+a_{1}[a_{1}-b_{1}]}(2n-b_{1}-a_{1}+a_{1}[a_{1}-b_{1}]\mid a_{k},\ldots,a_{2},a_{1}[a_{1}-b_{1}],b_{2},\ldots,b_{t})]\\
&=\chi[\mathfrak{q}^{C}_{2n-2b_{1}}(2n-2b_{1}\mid a_{k},\ldots,a_{2},a_{1}[a_{1}-b_{1}],a_{1}-b_{1}-a_{1}[a_{1}-b_{1}],b_{2},\ldots,b_{t})]\\
&=\chi[\mathfrak{q}^{C}_{n-b_{1}}(a_{1}-b_{1}-a_{1}[a_{1}-b_{1}],a_{1}[a_{1}-b_{1}],a_{2},\ldots,a_{k}\mid b_{2}\ldots,b_{t})]\\
\end{align*}
\end{proof}

\begin{lem}\label{lem3'}
Soient $\underline{a}=(a_{1},\ldots,a_{k})$ et $\underline{b}=(b_{1},\ldots,b_{t})$ deux compositions v\'erifiant $|\underline{b}|=|\underline{a}|=n$. Supposons $a_{1}>b_{1}$, on a :
$$\chi[\mathfrak{q}^{A}(\underline{a}\mid\underline{b})]=\chi[\mathfrak{q}^{A}(a_{1}-b_{1}-a_{1}[a_{1}-b_{1}],a_{1}[a_{1}-b_{1}],a_{2},\ldots,a_{k}\mid b_{2}\ldots,b_{t})]$$.
\end{lem}

\begin{proof}
R\'esulte du lemme pr\'ec\'edent et du corollaire \ref{cor1}(4).
\end{proof}

\begin{theo} Soient $(a,b,n)\in(\mathbb{N}^{\times})^{3}$ tel que $b\leq a\leq n$. Alors, l'indice de $\mathfrak{q}^{C}_{n}(a\mid b)$ est donn\'e par \\
$$\chi[\mathfrak{q}^{C}_{n}(a\mid b)]=\begin{cases} n\;si\;a=b\\
[\frac{a[a-b]}{2}]+[\frac{a-b-a[a-b]}{2}]+n-a\;si\;a\neq b\\
\end{cases}$$

\end{theo}
\begin{proof}

Supposons $a=b$, le r\'esultat se d\'eduit de l'assertion 4) du corollaire \ref{cor1}. Supposons $a>b$, il suit du corolaire \ref{cor1}, qu'on peut supposer $a=n$. Utilisons le lemme \ref{lem3}, on a
$$\chi[\mathfrak{q}^{C}_{n}(a\mid b)]=\chi[\mathfrak{q}^{C}_{a-b}(a-b-a[a-b],a[a-b]\mid\varnothing)$$ 
D'o\`u le r\'esultat.
\end{proof}

\begin{lem}\label{lemP}
Soit $\mathfrak{q}^{C}_{n}(n\mid\underline{a})$ une sous-alg\`ebre biparabolique de $\mathfrak{sp}(2n)$ o\`u $\underline{a}=(a_{1},\ldots,a_{k})$ est une composition v\'erifiant $|\underline{a}|\leq n$. Soient $s=n-|\underline{a}|$ et $\underline{a}^{'}=(a_{1},\ldots,a_{k},s)$.Consid\'erons $\mathfrak{q}^{A}(n\mid\underline{a}^{'})$ la sous-alg\`ebre biparabolique de $\mathfrak{gl}(n)$ associ\'ee \`a la paire de compositions $(n,\underline{a}^{'})$. Alors, il existe $\alpha\in\mathbb{N}$ et une composition $\underline{c}=(c_{1},\ldots,c_{j})$ v\'erifiant $j\leq k$ et $|\underline{c}|\leq s$ tels que l'on a
$$\chi[\mathfrak{q}^{C}_{n}(n\mid\underline{a})]=\alpha+\chi[\mathfrak{q}^{C}_{s+|\underline{c}|}(s+|\underline{c}|\mid\underline{c})]$$
$$\chi[\mathfrak{q}^{A}(n\mid\underline{a}^{'})]=\alpha+\chi[\mathfrak{q}^{A}(s+|\underline{c}|\mid\underline{c},s)]$$
\end{lem}

\begin{proof}
Si $2|\underline{a}|\leq n$, il suffit de consid\'erer $\alpha=0$ et $\underline{c}=\underline{a}$. Supposons que $2|\underline{a}|>n$, il suit du lemme 2.4 \cite{MB1} qu'il existe $1\leq i \leq k$ tel que $a_{i}\geq |d_{i}|$ (voir la d\'efinition des $d_{i}$ dans les th\'eor\`emes \ref{thm0} et \ref{thm0'}). Il r\'esulte alors des th\'eor\`emes \ref{thm0} et \ref{thm0'} que le r\'esultat s'obtient par r\'ecurrence sur $|\underline{a}|$.

\end{proof}

\begin{theo}

Soient $\underline{a}=(a_{1},\ldots,a_{k})$ et $\underline{b}=(b_{1},\ldots,b_{t})$ deux compositions v\'erifiant $|\underline{b}|\leq|\underline{a}|=n$, et $s=|\underline{a}|-|\underline{b}|$. Supposons $k+t<s$, alors $\mathfrak{q}^{C}_{n}(\underline{a}\mid\underline{b})$ n'est pas une sous-alg\`ebre de Frobenius de $\mathfrak{sp}(2n)$.

\end{theo}

\begin{proof}
Utilisant le lemme \ref{lem0'} et le th\'eor\`eme \ref{thm0'}, on a :\\
$$\chi[\mathfrak{q}^{C}_{n}(\underline{a}\mid\underline{b})]=\chi[\mathfrak{q}^{C}_{2n-a_{1}}(2n-a_{1}\mid a_{k},\ldots,a_{2},\underline{b})]$$
Supposons que la sous-alg\`ebre $\mathfrak{q}^{C}_{n}(\underline{a}\mid\underline{b})$ de $\mathfrak{sp}(2n)$ est de Frobenius. Il suit du lemme \ref{lemP} qu'il existe une composition $\underline{c}=(c_{1},\ldots,c_{j})$ telle que $j\leq k+t-1<s-1$, $|\underline{c}|\leq s$ et $\chi[\mathfrak{q}^{C}_{s+|\underline{c}|}(s+|\underline{c}|\mid \underline{c})]=0$. Comme $|\underline{c}|\leq [\frac{s+|\underline{c}|}{2}]$, il suit du lemme \ref{lemM} que $c_{i}=1,\;1\leq i\leq j$ et qu'il existe $\epsilon\in\{0,1\}$ tel que $j=|\underline{c}|=s-\epsilon$. On en d\'eduit que $j\geq s-1$ : une contradiction.
\end{proof}

\begin{theo}\label{I3} Soient $(a,b,c)\in\mathbb{N}^{\times}$. On pose $\max(a+b,c)=n$, $p=(a+b)\wedge(b+c)$ et $r=|a+b-c|$. On a :
\begin{itemize}
\item[1)] Si $p> r$, alors $\chi(\mathfrak{q}^{C}_{n}(a,b\mid c))=p-r+[\frac{r}{2}]$.
\item[2)] Si $p\leq r$, on a :\\
$\chi[\mathfrak{q}^{C}_{n}(a,b\mid c)]=[\frac{r}{2}]$ si $p$ et $r$ de m\^eme parit\'e.\\
$\chi[\mathfrak{q}^{C}_{n}(a,b\mid c)]=[\frac{r}{2}]-1$ sinon. 
\end{itemize}

\end{theo}

\begin{proof}

Supposons $c\leq a+b=n$. Il suit du lemme \ref{lem0'} et du th\'eor\`eme \ref{thm0'}  que l'on a

 $$\chi[\mathfrak{q}^{C}_{n}(a,b\mid c)]=\chi[\mathfrak{q}^{C}_{2n}(2n\mid b,a,c)]=\chi[\mathfrak{q}^{C}_{b+c+r}(b+c+r\mid b,c)]$$ 
Supposons $r=0$, il suit du corollaire \ref{cor1} (4) et du th\'eor\`eme \ref{ind} que l'on a $$\chi[\mathfrak{q}^{C}_{b+c+r}(b+c+r\mid b,c)]=\chi[\mathfrak{q}^{A}(b+c\mid b,c)]=b\wedge c=p$$
 Supposons dans la suite que $r\neq 0$ et rappelons les deux propri\'et\'es suivantes qui nous seront utiles apr\`es,

$\mathcal{P}:$ Tout triplet $(x,y,z)\in\mathbb{N}^{3}$ v\'erifie l'une des conditions suivantes
\begin{itemize}
\item[i)]$x+y\leq z$
\item[ii)]$x\geq y+z$
\item[iii)]$y>|x-z|$
\end{itemize}

$\mathcal{P}^{'}:$ Soit $z\in\mathbb{N}^{\times}$,  pour toute paire $(x,y)\in\mathbb{N}^{2}$, il existe $(x^{'},y^{'})\in\mathbb{N}^{2}$ v\'erifiant 
 $$\begin{cases}(x^{'}+y^{'})\wedge(y^{'}+z)=(x+y)\wedge(y+z)=:q\\
 x^{'}=z\;et\;y^{'}=q-z\;si\;q> z\\
 x^{'}+y^{'}\leq z\;si\;q\leq z\\
\end{cases}$$ \\
tel que l'on a $\chi[\mathfrak{q}^{C}_{x+y+z}(x+y+z\mid x,y)]=\chi[\mathfrak{q}^{C}_{x^{'}+y^{'}+z}(x^{'}+y^{'}+z\mid x^{'},y^{'})]$

La propri\'et\'e $\mathcal{P}$ est \'evidente. Montrons la propri\'et\'e $\mathcal{P}^{'}$ par r\'ecurrence sur la valeur de $x+y$.

Supposons $x+y\leq z$, en particulier $q\leq z$. Le r\'esultat est vrai avec $x=x^{'}$ et $y=y^{'}$. 
Supposons $x+y>z$, il suit de la propri\'et\'e $\mathcal{P}$ que l'on a $x\geq y+z$ ou $y> |x-z|$. Si $x=z$, en particulier $q=x+y=y+z>z$. Le r\'esultat est encore vrai avec $x=x^{'}$ et $y=y^{'}$. Si $x\neq z$, il suit du th\'eor\`eme \ref{thm0'} que l'on a
$$
\chi[\mathfrak{q}^{C}_{x+y+z}(x+y+z\mid x,y)]=\begin{cases}\chi[\mathfrak{q}^{C}_{x[|y+z|]+y+z}(x[|y+z|]+y+z\mid x[|y+z|],y)]\;si\;x\geq y+z\\
\chi[\mathfrak{q}^{C}_{x+y[|x-z|]+z}(x+y[|x-z|]+z\mid x,y[|x-z|])]\;si\;y>|x-z|
\end{cases}$$
Remarquons que $(x[|y+z|]+y)\wedge(y+z)=(x+y[|x-z|])\wedge(y[|x-z|]+z)=q$. Il suffit alors d'appliquer l'hypoth\`ese de r\'ecurrence \`a la paire
 $$(x_{1},y_{1})=\begin{cases} (x[|y+z|],y)\;si\;x\geq y+z\\
(x,y[|x-z|])\;si\;y>|x-z|
\end{cases}$$
Il suit de ce qui pr\'ec\`ede qu'il existe $(b^{'},c^{'})\in\mathbb{N}^{2}$ v\'erifiant 
 $$\begin{cases}(b^{'}+c^{'})\wedge(c^{'}+r)=(b+c)\wedge(c+r)=p\\
 b^{'}=r\;et\;c^{'}=p-r\;si\;p> r\\
 b^{'}+c^{'}\leq r\;si\;p\leq r\\
\end{cases}$$ \\
tel que l'on a $\chi[\mathfrak{q}^{C}_{n}(a,b\mid c)]=\chi[\mathfrak{q}^{C}_{b^{'}+c^{'}+r}(b^{'}+c^{'}+r\mid b^{'},c^{'})]$.\\ 
Supposons $p> r$, en particulier $b'=r$ et $c^{'}=p-r$. Il suit du th\'eor\`eme \ref{thm0'} et du lemme \ref{lemM} que l'on a 
 $$\chi[\mathfrak{q}^{C}_{n}(a,b\mid c)]=\chi[\mathfrak{q}^{C}_{b^{'}+c^{'}+r}(b^{'}+c^{'}+r\mid b^{'},c^{'})]=c^{'}+[\frac{r}{2}]=p-r+[\frac{r}{2}]$$
Supposons $p\leq r$, en particulier $b^{'}+c^{'}\leq r$. Il suit du lemme \ref{lemM} que l'on a 
  $$\chi[\mathfrak{q}^{C}_{n}(a,b\mid c)]=[\frac{b^{'}}{2}]+[\frac{c^{'}}{2}]+[\frac{r-b^{'}-c^{'}}{2}]$$
On distingue deux cas\\
Si $p$ est pair, les entiers  $c^{'},b^{'}\;et\;r$ sont alors de m\^eme parit\'e. En particulier, on a
 $$\chi[\mathfrak{q}^{C}_{n}(a,b\mid c)]=\begin{cases}[\frac{r}{2}]\;si\;r\;est\;pair\\
[\frac{r}{2}]-1\;si\;r\;est\;impair\\
\end{cases}$$ 
Si $p$ est impair, alors parmi les entiers $c^{'},b^{'}\;et\;r-b^{'}-c^{'}$, il y a un entier impair et un autre pair. En particulier, on a
$$\chi[\mathfrak{q}^{C}_{n}(a,b\mid c)]=\begin{cases}[\frac{r}{2}]-1\;si\;r\;est\;pair\\
[\frac{r}{2}]\;si\;r\;est\;impair\\
\end{cases}$$ 
Supposons $ a+b\leq c=n$. Il suit du corollaire \ref{cor1} que l'on a $\chi[\mathfrak{q}^{C}_{n}(a,b\mid c)]=\chi[\mathfrak{q}^{C}_{n}(c\mid a,b)]$. Remarquons que $p=(a+b)\wedge (b+r)$, le r\'esultat se d\'eduit de ce qui pr\'ec\`ede.
\end{proof}

\begin{cor}\label{J3} Soit $(a,b,c)\in\mathbb{N}^{\times}$, $\mathfrak{q}^{C}_{n}(a,b\mid c)$ est une sous-alg\`ebre de Frobenius de $\mathfrak{sp}(2n)$ si et seulement si $\max(a+b,c)=n$, et de plus, l'une des conditions suivantes est v\'erifi\'ee :
\begin{itemize}
\item[1)]$r=1$ et $p=1$
\item[2)]$r=2$ et $p=1$
\item[3)]$r=3$ et $p=2$
\end{itemize} 

\end{cor}

\begin{theo}
Soient $\underline{a}=(a_{1},\ldots,a_{k})$ et $\underline{b}=(b_{1},\ldots,b_{t})$ deux compositions v\'erifiant $|\underline{b}|\leq|\underline{a}|=n$, et $s=|\underline{a}|-|\underline{b}|$. Consid\'erons $\mathfrak{q}^{A}(\underline{a}\mid \underline{b},s)$ la sous-alg\`ebre biparabolique de $\mathfrak{gl}(n)$ associ\'ee \`a la paire de compositions $(\underline{a},(b_{1},\ldots,b_{t},s))$. Alors,
\begin{itemize}

\item[1)]Il existe une composition $\underline{d}$ de $s$ qui v\'erifie :\\
$\chi[\mathfrak{q}^{C}_{n}(\underline{a}\mid \underline{b})]-\chi[\mathfrak{q}^{A}(\underline{a}\mid \underline{b},s)]=\chi[\mathfrak{q}^{C}_{s}(\underline{d}\mid \varnothing)]-\chi[\mathfrak{q}^{A}(\underline{d}\mid s)]$
\item[2)]
\begin{itemize}
\item[i)] Supposons $\chi[\mathfrak{q}^{A}(\underline{a}\mid \underline{b},s)]=1$, alors $\chi[\mathfrak{q}^{C}_{n}(\underline{a}\mid \underline{b})]=\begin{cases} [\frac{s}{2}]-1\;si\;s\;est\;pair\\
[\frac{s}{2}]\;si\;s\;est\;impair\\
\end{cases}$
\item[ii)]Supposons $\chi[\mathfrak{q}^{C}_{n}(\underline{a}\mid \underline{b})]=0$, alors $\chi[\mathfrak{q}^{A}(\underline{a}\mid \underline{b},s)]=[\frac{s+1}{2}]$
\end{itemize}
\end{itemize}
\end{theo}

\begin{proof}
\begin{itemize}
\item[1)]
Il suit des lemmes \ref{lem0} et \ref{lem0'} que l'on a
$$\chi[\mathfrak{q}^{A}(\underline{a}\mid \underline{b},s)]=\chi[\mathfrak{q}^{A}(2n\mid \underline{a}^{-1}, \underline{b},s)]$$
$$\chi[\mathfrak{q}^{C}_{n}(\underline{a}\mid \underline{b})]=\chi[\mathfrak{q}^{C}_{2n}(2n\mid \underline{a}^{-1}, \underline{b})]$$
o\`u $\underline{a}^{-1}=(a_{k},\ldots,a_{1})$. D'apr\`es le lemme \ref{lemP}, il existe $\alpha\in\mathbb{N}$ et une composition $\underline{c}=(c_{1},\ldots,c_{t})$ v\'erifiant $|\underline{c}|\leq s$ telle que l'on a
  $$\chi[\mathfrak{q}^{A}(2n\mid \underline{a}^{-1}, \underline{b},s)]=\alpha+\chi[\mathfrak{q}^{A}(s+|\underline{c}|\mid \underline{c},s)]$$
 $$\chi[\mathfrak{q}^{C}_{2n}(2n\mid \underline{a}^{-1}, \underline{b})]=\alpha+\chi[\mathfrak{q}^{C}_{s+|\underline{c}|}(s+|\underline{c}|\mid \underline{c})]$$
Soit $\underline{d}=\begin{cases}(s-|\underline{c}|,c_{t},\ldots,c_{1})\;si\;|\underline{c}|<s\\
(c_{t},\ldots,c_{1})\;si\;|\underline{c}|=s\\
\end{cases} $. Puisque $|\underline{c}|\leq s$, alors $|\underline{c}|+s -2(c_{1}+\cdots+c_{i-1})\geq 2c_{i},\;1\leq i\leq t$. En appliquant les lemmes \ref{lem3} et \ref{lem3'} avec $a_{1}=|\underline{c}|+s -2(c_{1}+\cdots+c_{i-1})$ et $b_{1}=c_{i},\;i=1,\ldots,t$, on obtient
 $$\chi[\mathfrak{q}^{A}(s+|\underline{c}|\mid \underline{c},s)]=\chi[\mathfrak{q}^{A}(\underline{d}\mid s)]$$
 $$\chi[\mathfrak{q}^{C}_{s+|\underline{c}|}(s+|\underline{c}|\mid \underline{c})]=\chi[\mathfrak{q}^{C}_{s}(\underline{d}\mid \varnothing)]$$
  
D'o\`u le r\'esultat. 
\item[2)]
\begin{itemize}
\item[i)]
Supposons $\chi[\mathfrak{q}^{A}(\underline{a}\mid \underline{b},s)]=1$, alors $\alpha=0$ et $\chi[\mathfrak{q}^{A}(\underline{d}\mid s)]=1$. Par suite $\chi[\mathfrak{q}^{C}_{n}(\underline{a}\mid \underline{b})]=\chi[\mathfrak{q}^{C}_{s}(\underline{d}\mid \varnothing)]=[\frac{c_{1}}{2}]+\dots+[\frac{c_{t}}{2}]+[\frac{s-|\underline{c}|}{2}]$. D'autre part, il suit du th\'eor\`eme \ref{thm1} que le m\'eandre $\Gamma^{A}(\underline{d}\mid s)$ de $\mathfrak{q}^{A}(\underline{d}\mid s)$ est un segment, ce qui implique que, parmi les entiers $c_{1},\ldots,c_{t},s-|\underline{c}|$ et $s$, il en existe uniquement deux qui sont impairs. En particulier, 
$$\chi[\mathfrak{q}^{C}_{n}(\underline{a}\mid \underline{b})]=\begin{cases} [\frac{s}{2}]-1\;si\;s\;est\;pair\\
[\frac{s}{2}]\;si\;s\;est\;impair\\
\end{cases}$$
\item[ii)]Supposons $\chi[\mathfrak{q}^{C}_{n}(\underline{a}\mid \underline{b})]=0$. Il suit du lemme \ref{lemM} que $c_{i}=1,\;1\leq i\leq r$, $|\underline{c}|=s-\epsilon,\;\epsilon\in\{0,1\}$. Il en r\'esulte que $\chi[\mathfrak{q}^{A}(\underline{a}\mid \underline{b},s)]=\chi[\mathfrak{q}^{A}(\underline{d}\mid s)]=[\frac{s+1}{2}]$. 
\end{itemize}
\end{itemize} 
 \end{proof}
  
\begin{cor}
Soient $\underline{a}=(a_{1},\ldots,a_{k})$ et $\underline{b}=(b_{1},\ldots,b_{t})$ deux compositions v\'erifiant $|\underline{b}|\leq|\underline{a}|= n$. Supposons que $s=|\underline{a}|-|\underline{b}|=1\;ou\;2$, alors $\mathfrak{q}^{C}_{n}(\underline{a}\mid \underline{b})$ est une sous-alg\`ebre de Frobenius de $\mathfrak{sp}(2n)$ si et seulement si $\mathfrak{q}^{A}(\underline{a}\mid \underline{b},s)\cap \mathfrak{sl}(n)$ est une sous-alg\`ebre de Frobenius de $\mathfrak{sl}(n)$.
\end{cor}

\section{sous-alg\`ebres biparaboliques de $\mathfrak{so}(p)$}

Comme nous l'avons vu au num\'ero 2, toute sous-alg\`ebre biparabolique de $\mathfrak{so}(2n+1)$ (resp.  $\mathfrak{so}(2n)$) est conjugu\'ee sous l'action du groupe adjoint connexe \`a l'une des $\mathfrak{q}^{B}_{n}(\underline{a}\mid \underline{b})$ (resp. $\mathfrak{q}^{D}_{n}(\underline{a}\mid \underline{b})$) o\`u $\underline{a}$ et $\underline{b}$ sont deux compositions telles que $|\underline{a}|\leq n$ et $|\underline{b}|\leq n$. On associe \`a $\mathfrak{q}^{B}_{n}(\underline{a}\mid \underline{b})$ le m\^eme m\'eandre associ\'e \`a la sous-alg\`ebre biparabolique $\mathfrak{q}_{n}^{C}(\underline{a}\mid \underline{b})$ de $\mathfrak{sp}(2n)$. De plus, l'indice de $\mathfrak{q}^{B}(\underline{a}\mid \underline{b})$ est \'egal \`a celui de $\mathfrak{q}_{n}^{C}(\underline{a}\mid \underline{b})$ (voir \cite{meander.C}). Ainsi, tous les r\'esultats obtenus dans cet article pour les sous-alg\`ebres biparaboliques de $\mathfrak{sp}(2n)$ sont encore vrais pour les sous-alg\`ebres biparaboliques de $\mathfrak{so}(2n+1)$.\\

 Soit $\mathbb{C}^{2n}$ muni de la forme quadratique non d\'eg\'en\'er\'ee canonique. Comme expliqu\'e dans \cite[5.24]{DKT}, un sous-espace isotrope de dimension $n-1$, $E$ de $\mathbb{C}^{2n}$, est contenu exactement dans deux sous-espaces isotropes de dimension $n$, chacun desquels est laiss\'e invariant par le stabilisateur de $E$ dans $SO(2n)$. Il en r\'esulte que dans le cas de type $D$, on peut supposer que si $|\underline{a}|=n$ (resp. $|\underline{b}|=n$), alors $a_{k}>1$ (resp. $b_{t}>1$), ce que nous faisons d\'esormais.\\
 
Soit $\Xi_{n}$ l'ensemble des paires de compositions $(\underline{a}=(a_{1},\ldots,a_{k}),\underline{b}=(b_{1},\ldots,b_{t}))$ qui v\'erifient : $|\underline{a}|=n$, $|\underline{b}|=n-1$ et $a_{k}>1$ ou $|\underline{b}|=n$, $|\underline{a}|=n-1$ et $b_{t}>1$.

	Dans \cite{meander.D}, Panyushev et Yakimova associent \`a chaque sous-alg\`ebre biparabolique $\mathfrak{q}^{D}_{n}(\underline{a}\mid \underline{b})$ de $\mathfrak{so}(2n)$ un m\'eandre de la mani\`ere suivante :\\
	* Si $(\underline{a},\underline{b})\notin\Xi_{n}$, on associe \`a la sous-alg\`ebre biparabolique $\mathfrak{q}^{D}_{n}(\underline{a}\mid \underline{b})$ le m\'eandre $\Gamma^{D}_{n}(\underline{a}\mid \underline{b}):=\Gamma^{C}_{n}(\underline{a}\mid \underline{b})$, celui associ\'e \`a la sous-alg\`ebre biparabolique $\mathfrak{q}^{C}_{n}(\underline{a}\mid \underline{b})$ de $\mathfrak{sp}(2n)$.\\
	* Si $(\underline{a}=(a_{1},\ldots,a_{k}),\underline{b}=(b_{1},\ldots,b_{t}))\in\Xi_{n}$ et $|\underline{a}|= n$, on pose $\underline{b}^{'}:=(b_{1},\ldots,b_{t-1},b_{t}+1)$ et on associe \`a la sous-alg\`ebre biparabolique $\mathfrak{q}^{D}_{n}(\underline{a}\mid \underline{b})$ le m\'eandre $\Gamma^{D}_{n}(\underline{a}\mid \underline{b})$ obtenu du m\'eandre $\Gamma^{C}_{n}(\underline{a}\mid \underline{b}^{'})$ en rempla\c{c}ant l'arc joignant les sommets $a_{1}+\cdots+a_{k-1}+1$ et $n$ par un arc joignant les sommets $a_{1}+\cdots+a_{k-1}+1$ et $n+1$, et l'arc joignant les sommets $n+1$ et $n+a_{k}$ par un arc joignant les sommets $n$ et $n+a_{k}$. Ces deux nouveaux arcs sont les seuls arcs de $\Gamma^{D}_{n}(\underline{a}\mid \underline{b})$ qui se croisent, on les appelle arcs crois\'es. De plus, ils sont soit dans un m\^eme cycle, soit chacun d'eux est dans un segment distinct du segment qui contient l'autre arc. Si $|\underline{b}|= n$, on associe \`a la sous-alg\`ebre biparabolique $\mathfrak{q}^{D}_{n}(\underline{a}\mid \underline{b})$ le m\'eandre $\Gamma^{D}_{n}(\underline{a}\mid \underline{b})$ sym\'etrique de $\Gamma^{D}_{n}(\underline{b}\mid \underline{a})$ par rapport \`a la droite qui relie ses sommets.
\begin{Ex}$\\$

\medskip
$\\$
$\\$
$\Gamma_{10}^{C}(1,6,3\mid 3,2,5)=
{\setlength{\unitlength}{0.021in}
\raisebox{-11\unitlength}{%
\begin{picture}(50,-10)(-10,-10)
\multiput(0,3)(10,0){20}{\circle*{2}}

\put(70,4){\oval(40,25)[t]}
\put(70,4){\oval(20,15)[t]}
\put(120,4){\oval(40,25)[t]}
\put(120,4){\oval(20,15)[t]}

\put(10,4){\oval(20,15)[t]}
\put(35,4){\oval(10,10)[t]}

\put(180,4){\oval(20,15)[t]}
\put(155,4){\oval(10,10)[t]}

\put(35,3){\oval(50,30)[b]}
\put(35,3){\oval(30,20)[b]}
\put(35,3){\oval(10,10)[b]}

\put(80,3){\oval(20,15)[b]}

\put(155,3){\oval(50,30)[b]}
\put(155,3){\oval(30,20)[b]}
\put(155,3){\oval(10,10)[b]}

\put(110,2){\oval(20,15)[b]}
\put(95,-20){\line(0,1){50}} 
\end{picture}
} }
$ 
$\\$
\\
\\
\\
$\put(257,0){$\textcolor{blue}{\downarrow}$}$
\\
\\
\\
\medskip$\\$
$\Gamma_{10}^{D}(1,6,3\mid 3,2,4)=
{\setlength{\unitlength}{0.021in}
\raisebox{-11\unitlength}{%
\begin{picture}(50,-10)(-10,-10)
\multiput(0,3)(10,0){20}{\circle*{2}}

\put(70,4){\oval(40,25)[t]}
\put(70,4){\oval(20,15)[t]}
\put(120,4){\oval(40,25)[t]}
\put(120,4){\oval(20,15)[t]}

\put(10,4){\oval(20,15)[t]}
\put(35,4){\oval(10,10)[t]}

\put(180,4){\oval(20,15)[t]}
\put(155,4){\oval(10,10)[t]}

\put(35,3){\oval(50,30)[b]}
\put(35,3){\oval(30,20)[b]}
\put(35,3){\oval(10,10)[b]}

\put(85,3){\oval(30,20)[b]}

\put(155,3){\oval(50,30)[b]}
\put(155,3){\oval(30,20)[b]}
\put(155,3){\oval(10,10)[b]}

\put(105,3){\oval(30,20)[b]}
\put(95,-20){\line(0,1){50}}
\end{picture}
} }
$ 
\end{Ex}
$\\$
\begin{Ex}
$\\$
\medskip$\\$
$\Gamma_{5}^{C}(5\mid 5)=
{\setlength{\unitlength}{0.021in}
\raisebox{-11\unitlength}{%
\begin{picture}(110,-10)(-4,-10)
\multiput(0,3)(10,0){10}{\circle*{2}}

\put(20,4.5){\oval(40,25)[t]}
\put(20,4.5){\oval(20,15)[t]}
\put(70,4.5){\oval(40,25)[t]}
\put(70,4.5){\oval(20,15)[t]}

\put(20,4){\oval(40,25)[b]}
\put(20,4){\oval(20,15)[b]}
\put(70,4){\oval(40,25)[b]}
\put(70,4){\oval(20,15)[b]}

\end{picture}
} }
\put(-96.5,-30){\line(0,1){65}}$ 
$\put(-20,0){$\textcolor{blue}{\rightarrow}$}$
$\Gamma_{5}^{D}(4\mid 5)=
{\setlength{\unitlength}{0.021in}
\raisebox{-11\unitlength}{%
\begin{picture}(50,-10)(-4,-10)
\multiput(0,3)(10,0){10}{\circle*{2}}

\put(25,4.5){\oval(50,25)[t]}
\put(20,4.5){\oval(20,15)[t]}
\put(65,4.5){\oval(50,25)[t]}
\put(70,4.5){\oval(20,15)[t]}

\put(20,4){\oval(40,25)[b]}
\put(20,4){\oval(20,15)[b]}
\put(70,4){\oval(40,25)[b]}
\put(70,4){\oval(20,15)[b]}

\end{picture}
} }
\put(-5.5,-30){\line(0,1){65}}$

\end{Ex}

\begin{theo}\label{thm p}\cite{meander.D}
Soit $\mathfrak{q}^{D}_{n}(\underline{a}\mid \underline{b})$ une sous-alg\`ebre biparabolique de $\mathfrak{so}(2n)$, on a 
$$\chi[\mathfrak{q}^{D}_{n}(\underline{a}\mid \underline{b})]=nombre\;de\;cycles\;+\frac{1}{2}\;(nombre\;de\;segments\;non\;invariants)+\epsilon$$
o\`u $\epsilon$ est donn\'e de la mani\`ere suivante:$\\$
* Si $(\underline{a},\underline{b})\notin\Xi_{n}$, alors \\
$ \epsilon=\begin{cases} 0\;si\;|\underline{a}|-|\underline{b}|\;\text{est un entier pair} \\
1\;si\;|\underline{a}|-|\underline{b}|\;\text{ est un entier impair},\;\max(|\underline{a}|,|\underline{b}|)=n\;\text{et de plus l'arc joignant les sommets}\\
 \;\;\;\text{n et n+1 est un arc d'un\ segment de }\Gamma^{D}_{n}(\underline{a}\mid \underline{b})\\
-1\;\text{dans les cas restants}\\
\end{cases}$\\
* Si $(\underline{a},\underline{b})\in\Xi_{n}$, alors \\
$ \epsilon=\begin{cases} -1\;\text{si les arcs crois\'es sont dans un m\^eme cycle}\;de\;\Gamma_{n}^{D}(\underline{a}\mid\underline{b})\\
0\;sinon\\
\end{cases}$
\end{theo}

\begin{rem}
Avec les notations du th\'eor\`eme \ref{thm p}, si $(\underline{a},\underline{b})\notin\Xi_{n}$, on a
$$\chi[\mathfrak{q}^{D}_{n}(\underline{a}\mid \underline{b})]=\chi[\mathfrak{q}^{C}_{n}(\underline{a}\mid \underline{b})]+\epsilon$$
\end{rem}

De mani\`ere analogue au cas $\mathfrak{sp}(2n)$, on a le lemme suivant, 

\begin{lem}\label{lem r}

Soient $\underline{a}=(a_{1},\dots, a_{k})$ et $\underline{b}=(b_{1},\ldots,b_{t})$ deux compositions telles que $|\underline{b}|\leq|\underline{a}|\leq n$. On pose $\underline{a}^{-1}=(a_{k},\dots, a_{1})$. On a :
\begin{itemize}
\item[1)]$\chi[\mathfrak{q}_{n}^{D}(\underline{a}\mid\underline{b})]=\chi[\mathfrak{q}_{n}^{D}(\underline{b}\mid\underline{a})]$
\item[2)]$\chi[\mathfrak{q}_{n}^{D}(\underline{a}\mid\underline{b})]=\chi[\mathfrak{q}_{n+|\underline{a}|}^{D}(2|\underline{a}|\mid\underline{a}^{-1},\underline{b})]$ 
\end{itemize}
\end{lem}

\begin{theo}\label{thm r1}
Soient $t\in\mathbb{N}^{\times}$ et $\underline{a}=(a_{1},\ldots,a_{k})$ une composition v\'erifiant $|\underline{a}|\leq t\leq n$ et $(t\mid\underline{a})\notin\Xi_{n}$. On pose $a_{k+1}=t-|\underline{a}|$ et $d_{i}=(a_{1}+\dots +a_{i-1})-(a_{i+1}+\dots+a_{k+1}),\;1\leq i\leq k$.
 \begin{itemize}
\item[1)] Pour tout $1\leq i\leq k$ tel que $d_{i}\neq 0$ et tout $\alpha\in\mathbb{Z}$ tel que $a_{i}+\alpha|d_{i}|\geq 0$, on a 
$$\chi[\mathfrak{q}^{D}_{n}(t\mid\underline{a})]=\chi[\mathfrak{q}^{D}_{n+\alpha |d_{i}|}(t+\alpha |d_{i}|\mid a_{1},\ldots,a_{i-1},a_{i}+ \alpha |d_{i}|,a_{i+1},\ldots,a_{k})]$$
En particulier, on a 
$$\chi[\mathfrak{q}^{D}_{n}(t\mid\underline{a})]=\chi[\mathfrak{q}^{D}_{n-a_{i}+a_{i}[|d_{i}|]}(t-a_{i}+a_{i}[|d_{i}|]\mid a_{1},\ldots,a_{i-1},a_{i}[|d_{i}|],a_{i+1},\ldots,a_{k})]$$ 
\item[2)] Pour tout $1\leq i\leq k$ tel que $d_{i}=0$, on a 
$$\chi[\mathfrak{q}^{D}_{n}(t\mid\underline{a})]=a_{i}+\chi[\mathfrak{q}^{D}_{n-a_{i}}(t-a_{i}\mid a_{1},\ldots,a_{i-1},a_{i+1}\ldots,a_{k})]$$
\end{itemize}
\end{theo}

\begin{proof}
Rappelons que dans ce cas $\Gamma^{D}_{n}(t\mid\underline{a})=\Gamma^{C}_{n}(t\mid\underline{a})$, il suit du th\'eor\`eme \ref {thm p} que l'on a 
$$\chi[\mathfrak{q}^{D}_{n}(t\mid\underline{a})]=\chi[\mathfrak{q}^{C}_{n}(t\mid\underline{a})]+\epsilon$$
o\`u $\epsilon$ est donn\'e par:$\\$
$ \epsilon=\begin{cases} 0\;\;\;\text{si $t-|\underline{a}|$ est un entier pair} \\
1\;\;\;\text{si $t-|\underline{a}|$ est un entier impair, $t=n$ et de plus l'arc de $\Gamma^{D}_{n}(t\mid \underline{a})$ joignant les sommets}\\
 \text{$\;\;\;\;\;n$ et $n+1$ est un arc d'un segment de $\Gamma^{D}_{n}(t\mid \underline{a})$}\\
-1\;\text{dans les cas restants}\\
\end{cases}$\\
Compte tenu du th\'eor\`eme \ref{thm0'}, il reste \'egalement \`a v\'erifier la condition sur l'arc joignant les sommets $n$ et $n+1$. Pour cela, on pose \\
$(n^{'}\mid t^{'}\mid \underline{a}^{'}):=\begin{cases}(n+\alpha |d_{i}|\mid t+\alpha |d_{i}|\mid a_{1},\ldots,a_{i-1},a_{i}+ \alpha |d_{i}|,a_{i+1},\ldots,a_{k})\;si\;a_{i}+\alpha|d_{i}|\geq 0\\
(n-a_{i}\mid t-a_{i}\mid a_{1},\ldots,a_{i-1},a_{i+1}\ldots,a_{k})\;\text{si}\;d_{i}= 0
\end{cases}$,\\
en particulier, $t^{'}-|\underline{a}^{'}|=t-|\underline{a}|$ et $(t^{'}\mid\underline{a}^{'})\notin\Xi_{n^{'}}$. Il s'ensuit que $\Gamma^{D}_{n^{'}}(t^{'}\mid \underline{a}^{'})=\Gamma^{C}_{n^{'}}(t^{'}\mid \underline{a}^{'})$. Remarquons que dans le cas $d_{i}= 0$, le m\'eandre $\Gamma^{C}_{n}(t\mid \underline{a})$ est r\'eunion disjointe des m\'eandres $\Gamma^{C}_{n^{'}}(t^{'}\mid \underline{a}^{'})$ et $\Gamma^{C}_{a_{i}}(a_{i}\mid a_{i})$, et dans le cas $a_{i}+\alpha|d_{i}|\geq 0$, le m\'eandre $\Gamma^{C}_{n^{'}}(t^{'}\mid \underline{a}^{'})$ est obtenu du m\'eandre $\Gamma^{C}_{n}(t\mid \underline{a})$ comme expliqu\'e dans la preuve du lemme \ref{lem1}. Il en r\'esulte que l'arc joignant les sommets $n$ et $n+1$ est un arc d'un segment de $\Gamma^{C}_{n}(t\mid \underline{a})$ si et seulement si l'arc joignant les sommets $n^{'}$ et $n^{'}+1$ est un arc d'un segment de $\Gamma^{C}_{n^{'}}(t^{'}\mid \underline{a}^{'})$.

\end{proof}

\begin{rem}
Soient $(\underline{a}=(a_{1},\ldots,a_{k}),\underline{b}=(b_{1},\ldots,b_{t}))\in\Xi_{n}$ tel que $|\underline{b}|=n-1$. On pose $\underline{b}^{'}:=(b_{1},\ldots,b_{t-1},b_{t}+1)$ et on consid\`ere $\Gamma^{A}(\underline{a}\mid\underline{b}^{'})$ le m\'eandre de la sous-alg\`ebre biparabolique $\mathfrak{q}^{A}(\underline{a}\mid\underline{b}^{'})$ de $\mathfrak{gl}(n)$ dont les sommets sont les $n$ premiers sommets du m\'eandre $\Gamma_{n}^{D}(\underline{a}\mid\underline{b})$. Alors, les arcs crois\'es du m\'eandre $\Gamma_{n}^{D}(\underline{a}\mid\underline{b})$ sont dans un m\^eme cycle si et seulement si le dernier sommet $($le somment n$)$ du m\'eandre $\Gamma^{A}(\underline{a}\mid\underline{b}^{'})$ appartient \`a un cycle de $\Gamma^{A}(\underline{a}\mid\underline{b}^{'})$.
\end{rem}

\begin{cor}\label{corD} Avec les notations pr\'ec\'edentes, on a \\
$\chi[\mathfrak{q}_{n}^{D}(\underline{a}\mid\underline{b})]=\begin{cases}\chi[\mathfrak{q}^{A}(\underline{a}\mid\underline{b}^{'})]\;si\;le\;somment\;n\;appartient\;\text{\`a un segment de }\Gamma^{A}(\underline{a}\mid\underline{b}^{'})\\
\chi[\mathfrak{q}^{A}(\underline{a}\mid\underline{b}^{'})]-2\;sinon\\
\end{cases}$\\
En particulier, pour $n\geq 1$, on a 
$\chi[\mathfrak{q}_{n}^{D}(n\mid n-1)]=|n-2|$\\
\end{cor}

Pour une sous-alg\`ebre biparabolique $\mathfrak{q}^{A}(\underline{a}\mid\underline{b})$ de $\mathfrak{gl}(n)$, on pose \\\
$\Psi[\mathfrak{q}^{A}(\underline{a}\mid\underline{b})]=\begin{cases}\chi[\mathfrak{q}^{A}(\underline{a}\mid\underline{b})]\;\text{si le somment $n$ appartient \`a un segment de }\Gamma^{A}(\underline{a}\mid\underline{b})\\
\chi[\mathfrak{q}^{A}(\underline{a}\mid\underline{b})]-2\;\text{sinon}\\
\end{cases}$\\

\begin{theo}\label{thm r2}
Soient $\underline{a}=(a_{1},\ldots,a_{k})$ une composition v\'erifiant $1\leq |\underline{a}|=n-1$ $(i.e \;(n\mid\underline{a})\in\Xi_{n})$. On pose $\underline{a}^{'}=(a^{'}_{1},\ldots,a^{'}_{k})=(a_{1},\ldots,a_{k-1},a_{k}+1)$, $d_{k}=-(a^{'}_{1}+\cdots+a^{'}_{k-1})$ et
$d_{i}=(a^{'}_{1}+\cdots+a^{'}_{i-1})-(a^{'}_{i+1}+\cdots+a^{'}_{k})$, $ 1\leq i\leq k-1$.
\begin{itemize}
\item[1)] Pour tout $1\leq i\leq k$ tel que $d_{i}\neq 0$ et tout $\alpha\in\mathbb{Z}$ tel que $a^{'}_{i}+\alpha|d_{i}|\geq 0$, on a 
$$\chi[\mathfrak{q}^{D}_{n}(n\mid a_{1},\ldots,a_{k})]=\Psi[\mathfrak{q}^{A}(n+\alpha|d_{i}|\mid a^{'}_{1},\ldots,a^{'}_{i-1},a^{'}_{i}+\alpha|d_{i}|,a^{'}_{i+1},\ldots,a^{'}_{k})]$$
En particulier, si on pose $t_{i}=a^{'}_{i}-a^{'}_{i}[|d_{i}|]$, on a 
$$\chi[\mathfrak{q}^{D}_{n}(n\mid a_{1},\ldots,a_{k})]=\Psi[\mathfrak{q}^{A}(n-t_{i}\mid a^{'}_{1},\ldots,a^{'}_{i-1},a^{'}_{i}[|d_{i}|],a^{'}_{i+1}\ldots,a^{'}_{k})]$$
\item[2)] Pour tout $1\leq i\leq k$ tel que $d_{i}=0$, on a 
$$\chi[\mathfrak{q}^{D}_{n}(n\mid a_{1},\ldots,a_{k})]=a_{i}+\Psi[\mathfrak{q}^{A}(n-a_{i}\mid a_{1},\ldots,a_{i-1},a_{i+1}\ldots,a_{k})]$$
\end{itemize}
\end{theo}

\begin{proof}

Soit $ 1\leq i\leq k$. On pose $(n^{'}\mid\underline{a}^{''})$ la paire de compositions donn\'ee par\\
 $(n^{'}\mid\underline{a}^{''})=\begin{cases}(n+\alpha|d_{i}|\mid a^{'}_{1},\ldots,a^{'}_{i-1},a^{'}_{i}+\alpha|d_{i}|,a^{'}_{i+1},\ldots,a^{'}_{k})\;\text{si}\;a^{'}_{i}+\alpha|d_{i}|\geq 0\\
 (n-a_{i}\mid a_{1},\ldots,a_{i-1},a_{i+1}\ldots,a_{k})\;\text{si}\;d_{i}=0\\
 \end{cases}$\\
 Consid\'erons $\Gamma^{A}(n\mid\underline{a}^{'})$ le m\'eandre de $\mathfrak{q}^{A}(n\mid\underline{a}^{'})$ et $\Gamma^{A}(n^{'}\mid\underline{a}^{''})$ le m\'eandre de $\mathfrak{q}^{A}(n^{'}\mid\underline{a}^{''})$ obtenu de $\Gamma^{A}(n\mid\underline{a}^{'})$ de la mani\`ere introduite dans \cite{MB1}. On v\'erifie que le dernier sommet de $\Gamma^{A}(n\mid\underline{a}^{'})$ appartient \`a un segment de $\Gamma^{A}(n\mid\underline{a}^{'})$ si et seulement si le dernier sommet de $\Gamma^{A}(n^{'}\mid\underline{a}^{''})$ appartient \`a un segment de $\Gamma^{A}(n^{'}\mid\underline{a}^{''})$. Il suit du th\'eor\`eme \ref{thm0} que $\Psi[\mathfrak{q}^{A}(n\mid\underline{a}^{'})]=\Psi[\mathfrak{q}^{A}(n^{'}\mid\underline{a}^{''})]$. Le r\'esultat se d\'eduit im\'ediatement du corollaire pr\'ec\'edent.
\end{proof}

\begin{rem}
Compte tenu du lemme \ref{lem r}, les th\'eor\`emes \ref{thm r1} et \ref{thm r2} donnent un algorithme de r\'eduction pour calculer l'indice des sous-alg\`ebres biparaboliques de $\mathfrak{so}(2n)$.
\end{rem}

\begin{Ex} Consid\'erons la sous-alg\`ebre biparabolique $\mathfrak{q}^{D}_{335}(218,15,102\mid 33,301)$ de $\mathfrak{so}(670)$. On v\'erifie que $(218,15,102\mid 33,301)\in\Xi_{335}$. Compte tenu du lemme \ref{lem r}, on a 
$$\mathfrak{q}^{D}_{335}(218,15,102\mid 33,301)=\chi[\mathfrak{q}^{D}_{670}(670\mid 102,15,218,33,301)]$$
Puis, on applique le th\'eor\`eme \ref{thm r1}, on a 
\begin{align*}
\chi[\mathfrak{q}^{D}_{670}(670\mid 102,15,218,33,301)]&=\Psi[\mathfrak{q}^{A}(670\mid 102,15,218,33,302)]\\
 &=\Psi[\mathfrak{q}^{A}(452\mid 102,15,33,302)]\\
 &=\Psi[\mathfrak{q}^{A}(152\mid 102,15,33,2)]\\
 &=\Psi[\mathfrak{q}^{A}(52\mid 2,15,33,2)]\\
 &=\Psi[\mathfrak{q}^{A}(22\mid 2,15,3,2)]\\
 &=\Psi[\mathfrak{q}^{A}(7\mid 2,3,2)]\\
 &=3+\Psi[\mathfrak{q}^{A}(4\mid 2,2)]\\
 &=3+\Psi[\mathfrak{q}^{A}(2\mid 2)]\\
 &=3+2-2\\
 &=3\\
\end{align*}
\end{Ex}

Consid\'erons la famille des sous-alg\`ebre biparaboliques de $\mathfrak{so}(2n)$ de la forme $\mathfrak{q}^{D}_{n}(a,b\mid c)$ o\`u $(a,b,c)\in(\mathbb{N}^{\times})^{3}$. Dans le cas o\`u $(a,b\mid c)\notin\Xi_{n}$, il est simple d'obtenir une formule pour l'indice de $\mathfrak{q}^{D}_{n}(a,b\mid c)$ \`a partir du th\'eor\`eme \ref{I3} et du th\'eor\`eme \ref{thm p}. Pour le cas $(a,b\mid c)\in\Xi_{n}$, il suit du lemme \ref{lem r} et du corollaire \ref{corD} qu'on peut supposer $(a,b\mid c)=(a,n-a-1\mid n)$. D'apr\`es le th\'eor\`eme \ref{ind}, on a $\chi[\mathfrak{q}^{A}(a,n-a\mid n)]=a\wedge n$. Il suit du lemme 3.3 de \cite{MB1} que le dernier sommet du m\'eandre $\Gamma^{A}(a,n-a\mid n)$ appartient \`a un cycle de $\Gamma^{A}(a,n-a\mid n)$ si et seulement si $a\wedge n\geq 2$. Ainsi, on a le th\'eor\`eme suivant

\begin{theo}\label{I'3} Soit $(a,n)\in(\mathbb{N}^{\times})^{2}$ tel que $a\leq n-2$, on a\\ 
$\chi[\mathfrak{q}_{n}^{D}(a,n-a-1\mid n)]=\chi[\mathfrak{q}_{n}^{D}(a,n-a\mid n-1)]=|(a\wedge n)-2|$
\end{theo}

\begin{cor}\label{J'3} Soit $(a,b,c)\in(\mathbb{N}^{\times})^{3}$, $p=(a+b)\wedge(b+c)$, $r=|a+b-c|$ et $q=a\wedge n$. Alors $\mathfrak{q}^{D}_{n}(a,b\mid c)$ est une sous-alg\`ebre de Frobenius de $\mathfrak{so}(2n)$ si et seulement si l'une des conditions suivantes est v\'erifi\'ee :
\begin{itemize}
\item[1)]$r=1$, $q=2$ et $\max(a+b,c)=n$
\item[1)]$r=1$, $p=1$ et $\max(a+b,c)=n-1$
\item[2)]$r=2$, $p=1$ et $\max(a+b,c)=n$
\item[3)]$r=3$, $p=2$ et $\max(a+b,c)=n-1$
\end{itemize} 

\end{cor}

\begin{theo}\label{th B}
On pose $\mathcal{F}^{A}_{n}:=\{\mathfrak{q}^{A}(\underline{a}\mid\underline{b})\subset\mathfrak{gl}(n)\;:\;\chi[\mathfrak{q}^{A}(\underline{a}\mid\underline{b})]=1\}$ et $\mathcal{F}^{D}_{n}:=\{\mathfrak{q}^{D}_{n}(\underline{a}\mid\underline{b})\subset\mathfrak{so}(2n)\;:\;(\underline{a}\mid\underline{b})\in\Xi_{n}\;et\;\chi[\mathfrak{q}^{D}_{n}(\underline{a}\mid\underline{b})]=0\}$. Soient $\underline{a}=(a_{1},\ldots,a_{m})$ et $\underline{b}=(b_{1},\ldots,b_{t})$ deux compositions de $n$ telles que $\mathfrak{q}^{A}(\underline{a}\mid\underline{b})\in\mathcal{F}^{A}_{n}$, alors les sous-alg\`ebres $\mathfrak{q}^{D}_{2n}(2a_{1},\ldots,2a_{m}\mid 2b_{1},\ldots,2b_{t-1},2b_{t}-1)$ et $\mathfrak{q}^{D}_{2n}(2a_{1},\ldots,2a_{m-1},2a_{m}-1\mid 2b_{1},\ldots,2b_{t})$
appartiennent \`a $\mathcal{F}^{D}_{2n}$, et toutes les sous-alg\`ebres de $\mathfrak{so}(2n)$ appartenant \`a $\mathcal{F}^{D}_{2n}$ sont ainsi obtenues.\\
En particulier, pour tout $n\geq 1$, on a $\mathcal{F}^{D}_{2n+1}=\varnothing\;et\;\sharp\mathcal{F}^{D}_{2n}=2\sharp\mathcal{F}^{A}_{n}$
\end{theo}

\begin{proof}
Soit $(\underline{c}=(c_{1},\ldots,c_{m}),\underline{d}=(d_{1},\ldots,d_{t}))\in\Xi_{n}$ tel que $|\underline{c}|=n$. On pose $\underline{d}^{'}=(d_{1},\ldots,d_{t-1},d_{t}+1)$, c'est une composition de $n$. Il r\'esulte du corollaire \ref{corD} que $\mathfrak{q}_{n}^{D}(\underline{c}\mid\underline{d})$ est une sous-alg\`ebre de Frobenuis de $\mathfrak{so}(2n)$ si et seulement si $\Gamma^{A}(\underline{c}\mid\underline{d}^{'})$ est un cycle. D'autre part, il r\'esulte du lemme 3.4 de \cite{MB1} que $\Gamma^{A}(\underline{c}\mid\underline{d}^{'})$ est un cycle si et seulement si $\chi[\mathfrak{q}^{A}(\underline{c}\mid\underline{d}^{'})]=c_{1}\wedge\ldots\wedge c_{m}\wedge d_{1}\wedge\ldots\wedge d_{t-1}\wedge(d_{t}+1)=2$. Il suit du th\'eor\`eme \ref{thm1} que $\chi[\mathfrak{q}^{A}(\underline{a}\mid\underline{b})]=1$ si et seulement si $\Gamma^{A}(\underline{a}\mid\underline{b})$ est un segment. On d\'eduit alors du lemme 2.6 de \cite{MB1} que l'application $\mathfrak{q}^{A}(\underline{a}\mid\underline{b})\longmapsto \mathfrak{q}^{A}(2a_{1},\ldots,2a_{m}\mid 2b_{1},\ldots,2b_{t})$ est une bijection de $\mathcal{F}^{A}_{n}$ sur l'ensemble des sous-alg\`ebres biparaboliques $\mathfrak{q}^{A}(\underline{a}\mid\underline{b})$ de $\mathfrak{gl}(2n)$ tel que $\Gamma^{A}(\underline{a}\mid\underline{b})$ est un cycle. En particulier, $\mathfrak{q}^{D}_{2n}(2a_{1},\ldots,2a_{m}\mid 2b_{1},\ldots,2b_{t-1},2b_{t}-1)$ et $\mathfrak{q}^{D}_{2n}(2a_{1},\ldots,2a_{m-1},2a_{m}-1\mid 2b_{1},\ldots,2b_{t})$ appartiennent \`a $\mathcal{F}^{D}_{2n}$, et toutes les sous-alg\`ebres de $\mathfrak{so}(2n)$ appartenant \`a $\mathcal{F}^{D}_{2n}$ sont ainsi obtenues. De plus, la condition $c_{1}\wedge\ldots\wedge c_{m}\wedge d_{1}\wedge\ldots\wedge d_{t-1}\wedge(d_{t}+1)=2$ montre que $\mathcal{F}^{D}_{n}=\varnothing$ si $n$ est impair.
\end{proof}

Dans \cite{MB1}, nous avons \'etudi\'e la famille des sous-alg\`ebres biparaboliques de $\mathfrak{gl}(n)$ de la forme $\mathfrak{q}^{A}(n\mid\underbrace{a,\ldots,a}_{m},b)$ o\`u $(a,b,m)\in(\mathbb{N}^{\times})^{3}$. L'indice d'une telle sous-alg\`ebre est donn\'e par la formule suivante
$$\chi[\mathfrak{q}^{A}(n\mid\underbrace{a,\ldots,a}_{m},b)]=(a\wedge b)\phi_{m}(\frac{a}{a\wedge b},\frac{b}{a\wedge b})$$
o\`u $\phi_{m}$ la fonction d\'efinie sur $I:=\lbrace(a,b)\in\mathbb{N^{\times}}^{2}\mid \text{$a$ est impair ou b est impair}\rbrace$ par :
$$\phi_{m}(a,b)=\begin{cases} [\frac{m}{2}]+1\;\text{si $a$ est impair et $b$ est impair}\\
[\frac{m+1}{2}]\;\text{si $a$ est impair et $b$ est pair}\\
1\;\text{si $a$ est pair et $b$ est impair}\\
\end{cases}$$
De plus, nous avons d\'ecrit explicitement le m\'eandre $\Gamma^{A}(n\mid\underbrace{a,\ldots,a}_{m},b)$ (voir lemme 3.3 \cite{MB1}). Il en r\'esulte que le dernier sommet de $\Gamma^{A}(n\mid\underbrace{a,\ldots,a}_{m},b)$ appartient \`a un segment de $\Gamma^{A}(n\mid\underbrace{a,\ldots,a}_{m},b)$ si et seulement si $(a\wedge b)=1$. On en d\'eduit le th\'eor\`eme suivant :

\begin{theo}\label{thm a} (Avec les notations pr\'ec\'edentes),
Soit $(a,b,m)\in(\mathbb{N}^{\times})^{3}$. On pose $n=ma+b+1$ et $p=a\wedge(b+1)$. Alors, $(n,(\underbrace{a,\ldots,a}_{m},b))\in\Xi_{n}$ et on a, 
\begin{itemize}
\item[1)] $\chi[\mathfrak{q}^{D}_{n}(n\mid\underbrace{a,\ldots,a}_{m},b)]=\begin{cases}\phi_{m}(a,b+1)\;si\;p=1\\
p\phi_{m}(\frac{a}{p},\frac{b+1}{p})-2\;si\;p\geq 2\\
\end{cases}$
\item[2)] $\mathfrak{q}^{D}_{n}(n\mid\underbrace{a,\ldots,a}_{m},b)$ est une sous-alg\`ebre de Frobenius de $\mathfrak{so}(2n)$ si et seulement si $p=2$ et de plus l'une des conditions suivantes est v\'erifi\'ee \\
\begin{itemize}
\item[(i)]m=1
\item[(ii)]$\frac{a}{2}$ est pair et $\frac{b+1}{2}$ est impair
\item[(iii)]$\frac{a}{2}$ est impair, $\frac{b+1}{2}$ est pair et $m=2$
\end{itemize}

\end{itemize}
\end{theo}

 \bibliographystyle{plain} 
\bibliography{biblio}

\begin{thebibliography}{1}

\bibitem{MB}
M.~Bouhani.
\newblock {Formes lin\'eaires de type r\'eductif et unipotent}.
\newblock {\em J. Lie Theory}, 29(1):143--179, 2019.

\bibitem{MB1}
M.~Bouhani.
\newblock {Sur l'indice des sous-alg\`ebres biparaboliques de gl(n)}.
\newblock {\em ArXiv}, 1904.10555, 2019.

\bibitem{Bourbaki}
N.~Bourbaki.
\newblock {\em {\'El\'ements de Math\'ematiques. Groupes et alg\`ebres de Lie.
  Chapitres 4,5 et 6, Masson, 1981}}.

\bibitem{C.meanders}
V.~Coll, M.~Hyatt, and C.~Magnant.
\newblock {Symplectic meanders}.
\newblock {\em Communication in Algebra}, 45(11):4717--4729, 2017.
\newblock doi: 10.1080/00927872.2016.1278012.

\bibitem{D.K}
V.~Dergachev and A.~A. Kirillov.
\newblock {Index of Lie algebras of seaweed type}.
\newblock {\em J. Lie Theory}, 10(2):331--343, 2010.

\bibitem{DKT}
M.~Duflo, M.~S. Khalgui, and P.~Torasso.
\newblock {Alg\`ebres de Lie quasi-r\'eductives}.
\newblock {\em Transform. Groups}, 17(2):417--470, 2012.

\bibitem{meander.C}
D.~I. Panyushev and O.~Yakimova.
\newblock {On seaweed subalgebras and Meander graphs in type C}.
\newblock {\em Pacific J. Math.}, 285(2):485--499, 2016.

\bibitem{meander.D}
D.~I. Panyushev and O.~Yakimova.
\newblock {On seaweed subalgebras and Meander graphs in type D}.
\newblock {\em Journal of Pure and Applied Algebra}, 222(11):3414--3431, 2018.

\end{thebibliography}

\end{document}